\documentclass[leqno,11pt]{amsart}
\usepackage{amsmath}
\usepackage{amssymb}
\usepackage{amsthm}
\usepackage{mathtools}
\usepackage{thmtools}
\usepackage{hhline}
\usepackage{tikz-cd}
\usepackage{enumitem}
\setlist[enumerate]{label={\normalfont (\roman*)}, itemsep=1ex}

\usepackage[utf8]{inputenc}
\usepackage[T1]{fontenc}
\usepackage{lmodern}
\usepackage[babel]{microtype}
\usepackage[english]{babel}
\usepackage{relsize}
\usepackage{calc}
\usepackage{multicol}

\usepackage{svg}

\usepackage{graphicx}
\usepackage{subcaption}

\counterwithin{figure}{section}

\usepackage{geometry}
\usepackage{xcolor} 
\colorlet{darkishRed}{red!60!black}
\colorlet{darkishBlue}{blue!60!black}
\colorlet{darkishGreen}{green!60!black}
\colorlet{darkblue}{blue!70!black}
\colorlet{darkishViolet}{violet}
\usepackage[linktoc=all]{hyperref}
\hypersetup{
	colorlinks,
    linkcolor={darkishBlue},
	citecolor={darkishGreen},
	urlcolor={darkishBlue}
}
\usepackage[nameinlink, capitalise, noabbrev]{cleveref}
\crefformat{enumi}{#2#1#3}
\crefformat{equation}{#2(#1)#3}
\crefname{equation}{}{}
\crefname{mainresult}{Theorem}{Theorems}
\let\setminus=\smallsetminus

\newcommand{\se}{\subseteq}
\newcommand{\sm}{\setminus}

\renewcommand{\leq}{\leqslant}
\renewcommand{\geq}{\geqslant}
\renewcommand{\ge}{\geq}
\renewcommand{\le}{\leq}

\DeclareMathOperator{\Aut}{Aut}

\newtheorem{theorem}{Theorem}[section] 
\newtheorem{proposition}[theorem]{Proposition}
\newtheorem{corollary}[theorem]{Corollary}
\newtheorem{lemma}[theorem]{Lemma}
\newtheorem{keylemma}[theorem]{Key Lemma}

\crefname{corlemma}{Correspondence}{Correspondences}

\crefname{liftlemma}{Lift}{Lifts}

\crefname{projlemma}{Projection}{Projections}

\newtheorem{observation}[theorem]{Observation}

\newtheorem{problem}[theorem]{Problem}

\newtheorem{mainresult}{Theorem} 

\newtheorem{maincorollary}[mainresult]{Corollary}

\theoremstyle{definition}
\newtheorem{example}[theorem]{Example}

\newtheorem{definition}[theorem]{Definition}

\theoremstyle{remark}

\newtheorem*{acknowledgment}{Acknowledgement}
\newtheorem*{claim*}{Claim}

\crefname{claim}{Claim}{Claims}
\crefname{enumi}{}{}
\AtEndEnvironment{proof}{\setcounter{claim}{0}}

\usepackage{etoolbox}

\newcommand{\COMMENT}[1]{{}}
\definecolor{cMaroon}{HTML}{93152a}

\newcommand{\defn}[1]{{\color{darkishRed}{\emph{#1}}}}
\newcommand{\casen}[1]{{\color{darkishViolet}{\emph{#1}}}}

\let\rho=\varrho
\let\phi=\varphi
\def\Z{\mathbb Z}

\makeatletter

\def\calCommandfactory#1{%
  \expandafter\def\csname c#1\endcsname{\mathcal{#1}}}
\def\frakCommandfactory#1{%
  \expandafter\def\csname frak#1\endcsname{\mathfrak{#1}}}
   
\newcounter{ctr}
\loop
  \stepcounter{ctr}
  \edef\X{\@Alph\c@ctr}
  \expandafter\calCommandfactory\X
  \expandafter\frakCommandfactory\X
\ifnum\thectr<26
\repeat

\makeatletter
\patchcmd{\@tocline}
  {\hfil}
  {\leaders\hbox{\,.\,}\hfil}
  {}{}
\makeatother

\newcounter{mylabelcounter}

\makeatletter
\newcommand{\labelText}[2]{%
#1\refstepcounter{mylabelcounter}%
\immediate\write\@auxout{%
  \string\newlabel{#2}{{1}{\thepage}{{\unexpanded{#1}}}{mylabelcounter.\number\value{mylabelcounter}}{}}%
}%
}

\newlist{defenum}{enumerate}{1}
\setlist[defenum]{label=(\upshape\thedefinition.\arabic*)}

\title[A characterisation of all vertex-transitive finite graphs of connectivity < 5]%
{A characterisation of all vertex-transitive finite graphs\\of connectivity < 5}

\author[Jan Kurkofka]{Jan Kurkofka${}^{\includegraphics[height=.7\baselineskip]{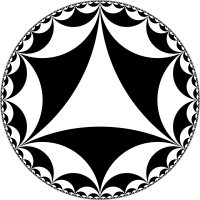}}$}
\author[Tim Planken]{Tim Planken$^\dagger$}

\thanks{$\includegraphics[height=.7\baselineskip]{Figures/MiniFarey.png}$ funded by the Deutsche Forschungsgemeinschaft (DFG, German Research Foundation) -- 566118291; 546892829}

\thanks{${}^\dagger$ funded by the Deutsche Forschungsgemeinschaft (DFG, German Research Foundation) -- 546892829}



\begin{document}
\thispagestyle{empty}

\begin{abstract}
    We characterise all vertex-transitive finite connected graphs
    as essentially 5-connected
    or on a short list of explicit graph-classes.
    Our proof heavily uses Tutte-type canonical decompositions.
\end{abstract}
\topskip0pt
\vspace*{\fill}
\maketitle

\vspace*{\fill}

\thispagestyle{empty}

\newpage
\setcounter{page}{1}

\section{Introduction}

We aim to solve the following problem for graphs of low connectivity:

\begin{problem}\label{IntroProblem}
    Explicitly characterise the vertex-transitive finite graphs.
\end{problem}

Droms, Servatius and Servatius~\cite{infiniteSPQR} solved \cref{IntroProblem} for graphs of connectivity~$< 3$.
Two years ago, Carmesin and Kurkofka~\cite{Tridecomp,TriAiC} extended their characterisation to connectivity~3, as follows.
A~graph~$G$ is commonly said to be \emph{arc-transitive} if for every two ordered pairs $(u_i,v_i)$ with $u_i v_i\in E(G)$ for $i=1,2$, there is an automorphism $\varphi$ of $G$ such that $\varphi(u_1)=u_2$ and $\varphi(v_1)=v_2$.
Grohe~\cite{grohe2016quasi} defines a graph $G$ to be \emph{quasi-4-connected} if it is 3-connected and every 3-separation $(A,B)$ of $G$ satisfies $|A\sm B|=1$ or $|B\sm A|=1$.
Let $H$ be an $r$-vertex graph and let $G$ be an $r$-regular graph.
We define an \emph{$H$-expansion} of $G$ to be a graph obtained from $G$ by replacing the vertices $v$ of $G$ with pairwise disjoint copies $H_v$ of~$H$ and by replacing every edge $e=uv$ of $G$ with an $H_u$--$H_v$ edge $\hat e$ so that $\hat e_1$ and $\hat e_2$ share no ends for every two $e_1\neq e_2\in E(G)$.
Thus, the edges $\hat e$ form a perfect matching of the $H$-expansion.

\begin{theorem}\label{intro:triVxCon}\cite[Cor.~2]{Tridecomp}\footnote{We slightly strengthen \cite[Cor.~2]{Tridecomp} in the present paper to obtain \cref{intro:triVxCon}. The use of $H$-expansions is new.}
    A finite connected graph $G$ is vertex-transitive if and only if $G$ is either
\begin{itemize}
    \item quasi-4-connected and vertex-transitive,
    \item the $K_3$-expansion of a quasi-4-connected 3-regular arc-transitive graph,
    \item a cycle, $K_2$ or $K_1$.
\end{itemize}
\end{theorem}

\cref{intro:triVxCon} leaves it open to gain new insights into vertex-transitive quasi-4-connectivity.
Recall that Tutte found an explicit way to uniquely decompose every 2-connected graph~$G$ along 2-separators into smaller parts that are either 3-connected, cycles or~$K_2$'s.
The way in which the parts fit together to form~$G$ is then captured by a unique tree-decomposition of~$G$, the \emph{Tutte-decomposition} of~$G$.
We may think of the Tutte-decomposition as a sophisticated generalisation of the block-cut tree to 2-connected graphs.
Both the block-cut tree and the Tutte-decomposition exhibit a rare combination of two properties that makes them into the perfect tools to approach \cref{IntroProblem}: both decompositions are \emph{explicit}, and they are \emph{canonical} in that they are invariant under all automorphisms of the graphs that they decompose.

This approach was limited to connectivity~$<3$ since the sixties when Tutte found his decomposition, until two years ago, when Carmesin and Kurkofka~\cite{Tridecomp,TriAiC} found a Tutte-type canonical decomposition for 3-connectivity and used it to prove~\cref{intro:triVxCon}.
However, their \emph{tri-decomposition} does not decompose quasi-4-connected graphs any further.
Indeed, new difficulties arise that obstruct all naive attempts at extending the tri-decomposition to quasi-4-connectivity; see the introduction of~\cite{TetraDecomp}.
As it was not known if a Tutte-type canonical decomposition for quasi-4-connectivity even exists, it was open how to extend \cref{intro:triVxCon} to higher connectivities.

Just recently, we have found a Tutte-type canonical decomposition for 4-connected graphs, the \emph{tetra-decomposition}~\cite{TetraDecomp}.
The tetra-decomposition allows us to approach \cref{IntroProblem} for 4-connected graphs.
However, if $G$ is a quasi-4-connected graph, then $G$ may have vertices of degree~3, and hence~$G$ can lie outside of the scope of the tetra-decomposition.
An established way of dealing with degree-3-vertices of graphs while preserving their non-trivial connectivity structure is to consecutively apply $Y$--$\Delta$ transformations.
In~\cite{TetraDecomp}, we introduced a canonical $Y$--$\Delta$ transformation $G\mapsto G^\Delta$ that makes every quasi-4-connected graph $G$ into a 4-connected graph~$G^\Delta$ (unless $G$ has $\le 6$ vertices).
This allows us to indirectly generalise the tetra-decomposition from 4-connectivity to quasi-4-connectivity: starting with a quasi-4-connected graph~$G$, we first apply the canonical $Y$--$\Delta$ operation to make $G$ into a 4-connected graph~$G^\Delta$, and then apply the tetra-decomposition to~$G^\Delta$.
As $G$ and $G^\Delta$ are similar, in practice we can use the tetra-decomposition of $G^\Delta$ to study~$G$.
Here, we will consider $G^\Delta$ only for 3-regular~$G$, in which case $G^\Delta$ defaults to the line-graph of~$G$.

In this paper, we use the above approach to solve \cref{IntroProblem} for quasi-4-connected graphs that are not 2-quasi-5-connected, see \cref{mainthm:transitive}.
Here, a graph~$G$ is \emph{2-quasi-5-connected} if~$G$ is quasi-4-connected and every 4-separation~$(A,B)$ of~$G$ satisfies $|A \sm B| \leq 2$ or $|B \sm A| \leq 2$.
Perhaps surprisingly, many new and rich graph classes occur in the characterisation.
Our proof reflects this: while \cref{intro:triVxCon} follows fairly quickly from having a closer look at the decompositions, the proof of \cref{mainthm:transitive} faced substantially new challenges and spans the rest of this paper.

To state \cref{mainthm:transitive}, we need the notion of graph-decompositions which generalise tree-decom\-posi\-tions in that the decomposition-tree may take the form of an arbitrary graph, such as a cycle, in which case we speak of a \emph{cycle-decomposition}~\cite{canonicalGraphDec}.
We call a graph $X$ a \emph{cycle of $X_1$-torsos, \dots , $X_k$-torsos, $Y_1$-bags, \dots, $Y_\ell$-bags} where the $X_i$ and $Y_j$ are graphs, if $X$ is quasi-4-connected and has a cycle-decomposition with all adhesion-sets of size~2 such that for every bag either its torso is isomorphic to some~$X_i$ or the bag itself is isomorphic to some~$Y_j$.
For more definitions regarding cycle-decompositions, see \cref{dfn:transitive-cycles} or the figures referenced in \cref{mainthm:transitive}.

\begin{restatable}{mainresult}{maintransitive}
    \label{mainthm:transitive}
    A quasi-4-connected finite graph~$G$ is vertex-transitive if and only if $G$ either is
    \begin{enumerate}
        \item\label{mainitm:transitive-1} a $(\ge 4)$-regular quasi-5-connected vertex-transitive graph,
        \item\label{mainitm:3reg2q5c} a 3-regular 2-quasi-5-connected vertex-transitive graph,
        \item\label{mainitm:transitive-2} a $K_4$-expansion of a quasi-5-connected 4-regular arc-transitive graph, or a vertex-transitive $C_4$-expansion of a quasi-5-connected 4-regular arc-transitive graph,
        \item\label{mainitm:transitive-3} a cycle of $K_4$-bags of length $\ge 4$ (\cref{fig:CycleOfK4s}),
        \item\label{mainitm:transitive-4} a cycle alternating between $K_4$-bags and $C_4$-bags of length $\ge 4$ (\cref{fig:CycleAlternating}),
        \item\label{mainitm:transitive-5} a cycle of triangle-bags of length $\ge 6$ (\cref{fig:cycleOfTriangles}),
        \item\label{mainitm:transitive-K22C4} a cycle alternating between $K_{2,2}$-bags and $C_4$-torsos of length~$\ge 6$ (\cref{fig:CycleAltK22C4}),
        \item\label{mainitm:transitive-C4} a cycle of $C_4$-bags of length~$\ge 4$ (\cref{fig:CycleC4sNormal}) or $K_3\square K_2$,
        \item\label{mainitm:transitive-6} a cycle of $K_{2,2}$-bags of length $\ge 3$, the $K_4$-expansion thereof, or a vertex-transitive $C_4$-expansion thereof (\cref{fig:CycleOfK22s}), or
        \item\label{mainitm:transitive-7} the line-graph of the cube, its $K_4$-expansion, or one of the three $C_4$-expansions of it depicted in \cref{fig:LineGraphCube}.
    \end{enumerate}
\end{restatable}

\newpage

\begin{figure}[ht]
\centering
\begin{minipage}{.3\textwidth}%
    \centering
    \captionsetup{width=1\textwidth}
    \includegraphics[height=7\baselineskip]{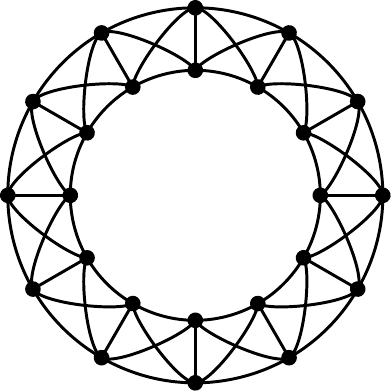}
    \captionof{figure}{{\,}\\A cycle of $K_4$-bags\\{\,}}
    \label{fig:CycleOfK4s}
\end{minipage}\hfill\begin{minipage}{.3\textwidth}
    \centering
    \captionsetup{width=1\textwidth}
    \includegraphics[height=7\baselineskip]{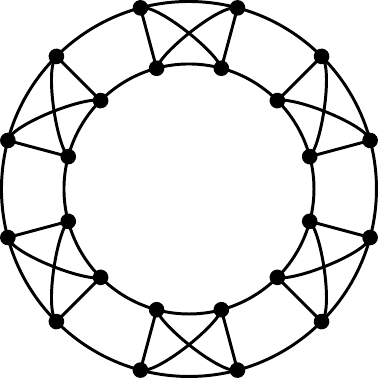}
    \captionof{figure}{{\,}\\A cycle alternating between $K_4$-bags and $C_4$-bags}
    \label{fig:CycleAlternating}
\end{minipage}\hfill%
\begin{minipage}{.3\textwidth}%
    \centering
    \captionsetup{width=1\textwidth}
    \includegraphics[height=7\baselineskip]{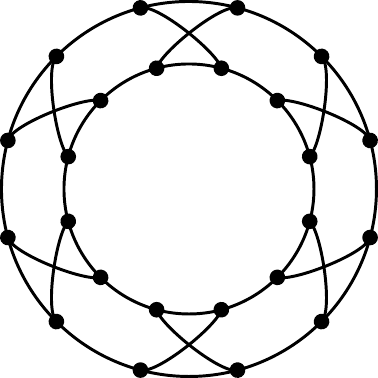}
    \captionof{figure}{{\,}\\A cycle alternating between $K_{2,2}$-bags and $C_4$-torsos}
    \label{fig:CycleAltK22C4}
\end{minipage}
\end{figure}

\begin{figure}[ht]
\centering
\begin{minipage}{.48\textwidth}
    \centering
    \captionsetup{width=1\textwidth}
    \includegraphics[height=7\baselineskip]{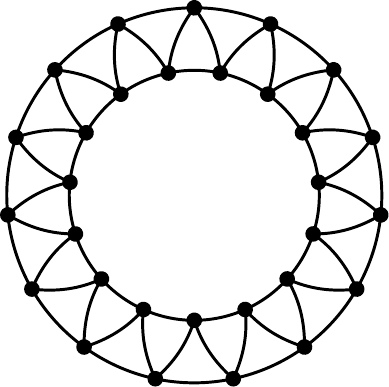}\hfill%
    \includegraphics[height=7\baselineskip]{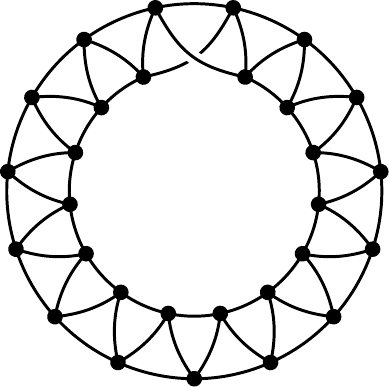}
    \captionof{figure}{{\,}\\Cycles of triangle-bags\\{\,}\\{\,}}
    \label{fig:cycleOfTriangles}
\end{minipage}\hfill\begin{minipage}{.48\textwidth}
    \centering
    \captionsetup{width=1\textwidth}
    \includegraphics[height=7\baselineskip]{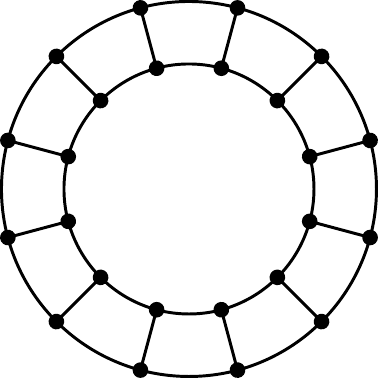}\hfill%
    \includegraphics[height=7\baselineskip]{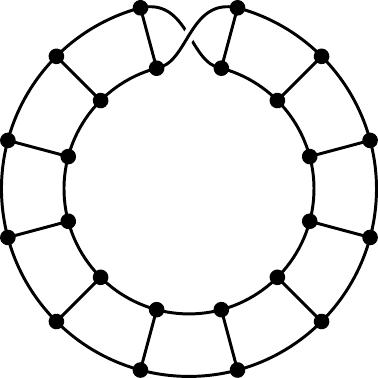}
    \captionof{figure}{{\,}\\Cycles of $C_4$-bags\\{\,}\\{\,}}
    \label{fig:CycleC4sNormal}
\end{minipage}
\end{figure}

\begin{figure}[ht]
    \centering
    \captionsetup{width=1\textwidth}
    \includegraphics[height=7\baselineskip]{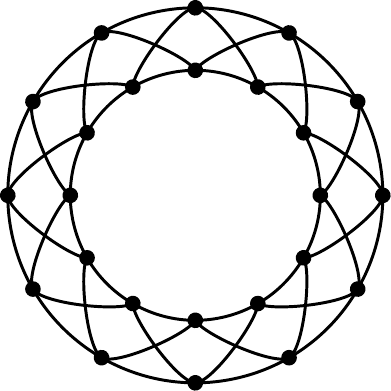}\hfill%
    \includegraphics[height=7\baselineskip]{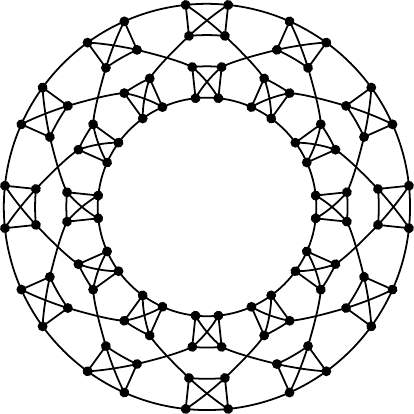}\hfill%
    \includegraphics[height=7\baselineskip]{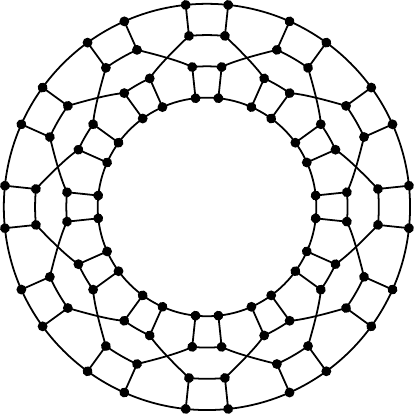}\hfill%
    \includegraphics[height=7\baselineskip]{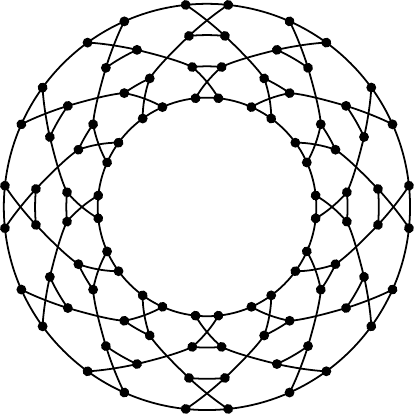}
    \captionof{figure}{A cycle of $K_{2,2}$-bags, its $K_4$-expansion, and two vertex-transitive $C_4$-expansions}
    \label{fig:CycleOfK22s}
\end{figure}

\begin{figure}[ht]
    \centering
    \captionsetup{width=1\textwidth}
    \includegraphics[height=5.5\baselineskip]{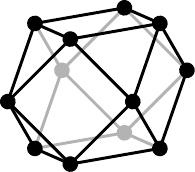}\hfill%
    \includegraphics[height=5.5\baselineskip]{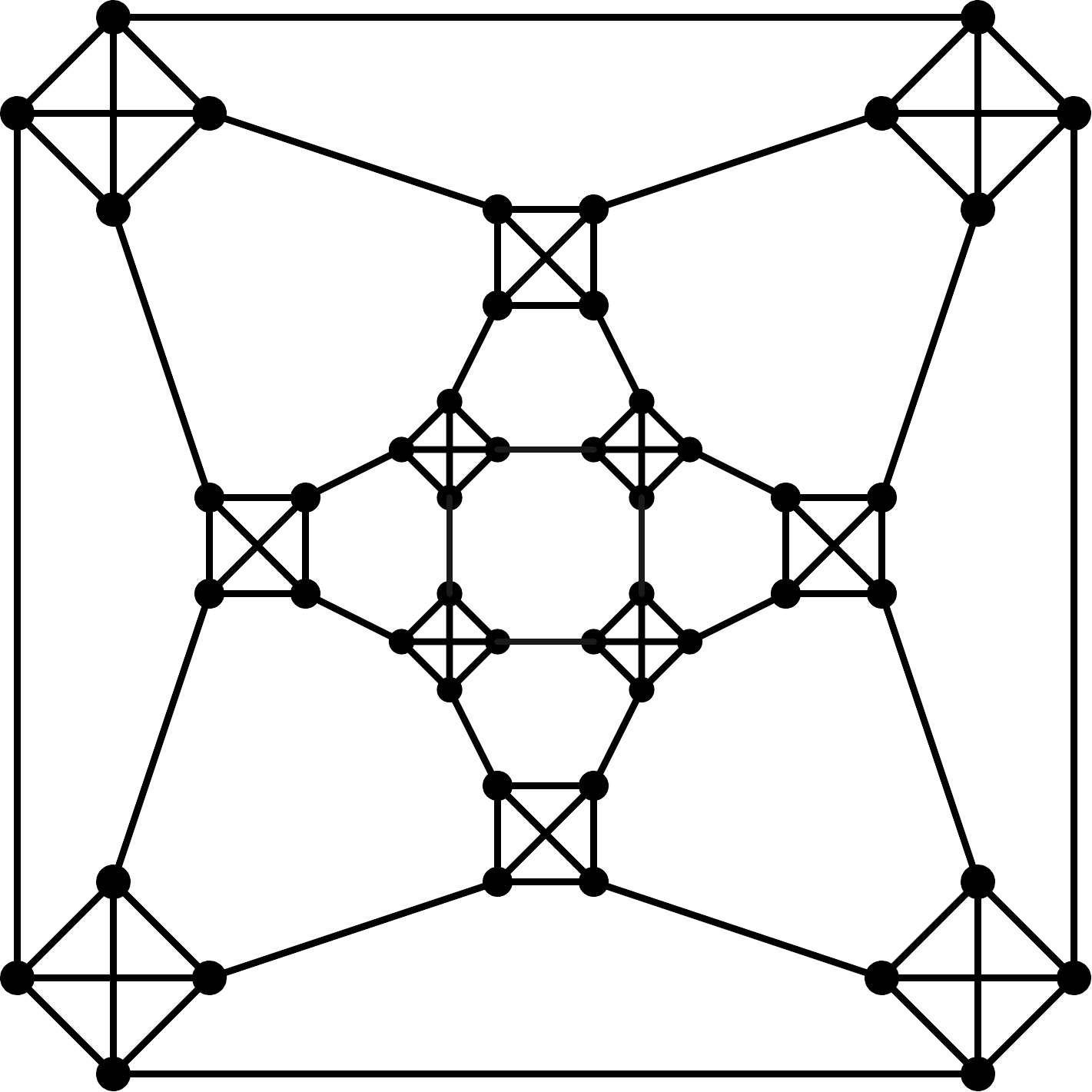}\hfill%
    \includegraphics[height=5.5\baselineskip]{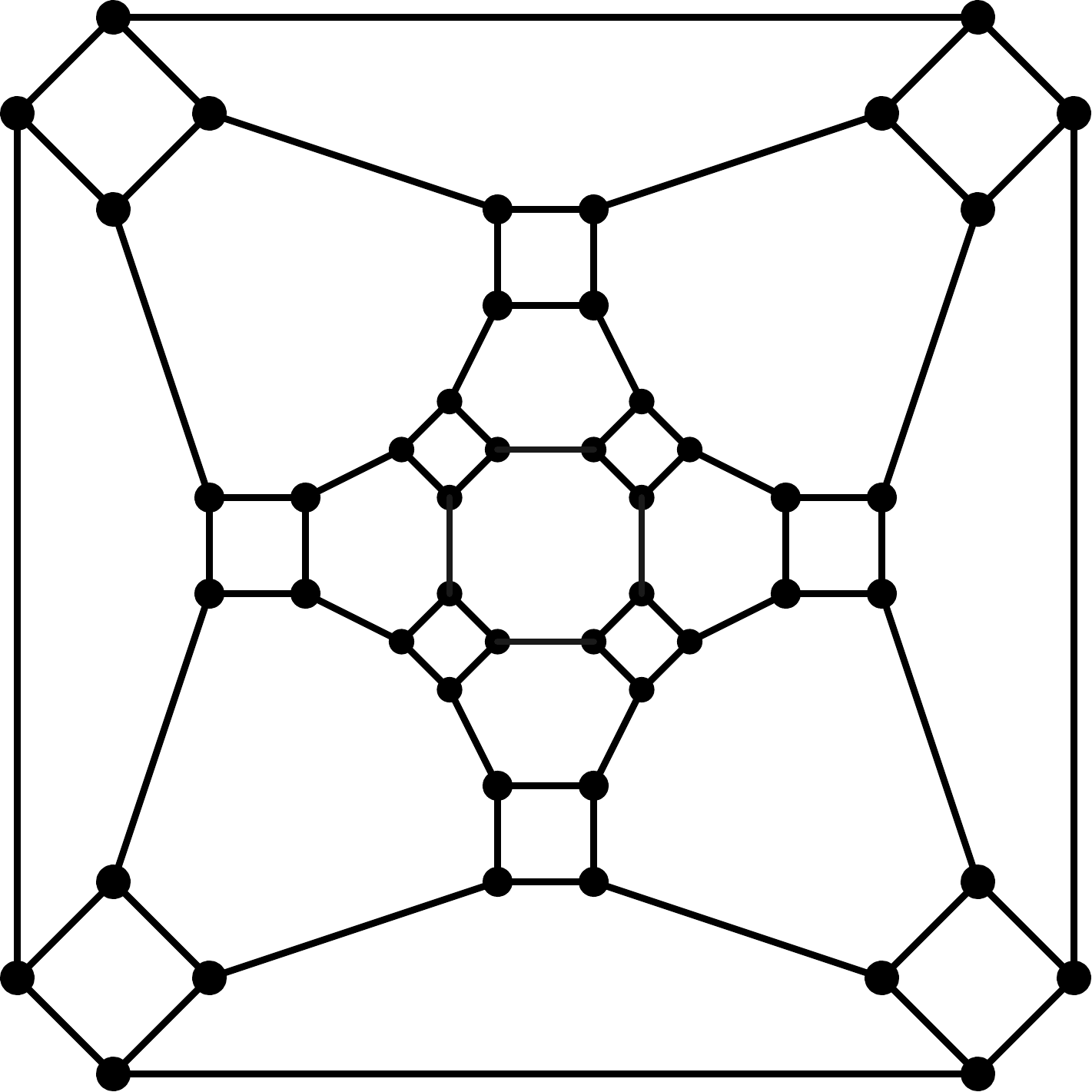}\hfill%
    \includegraphics[height=5.5\baselineskip]{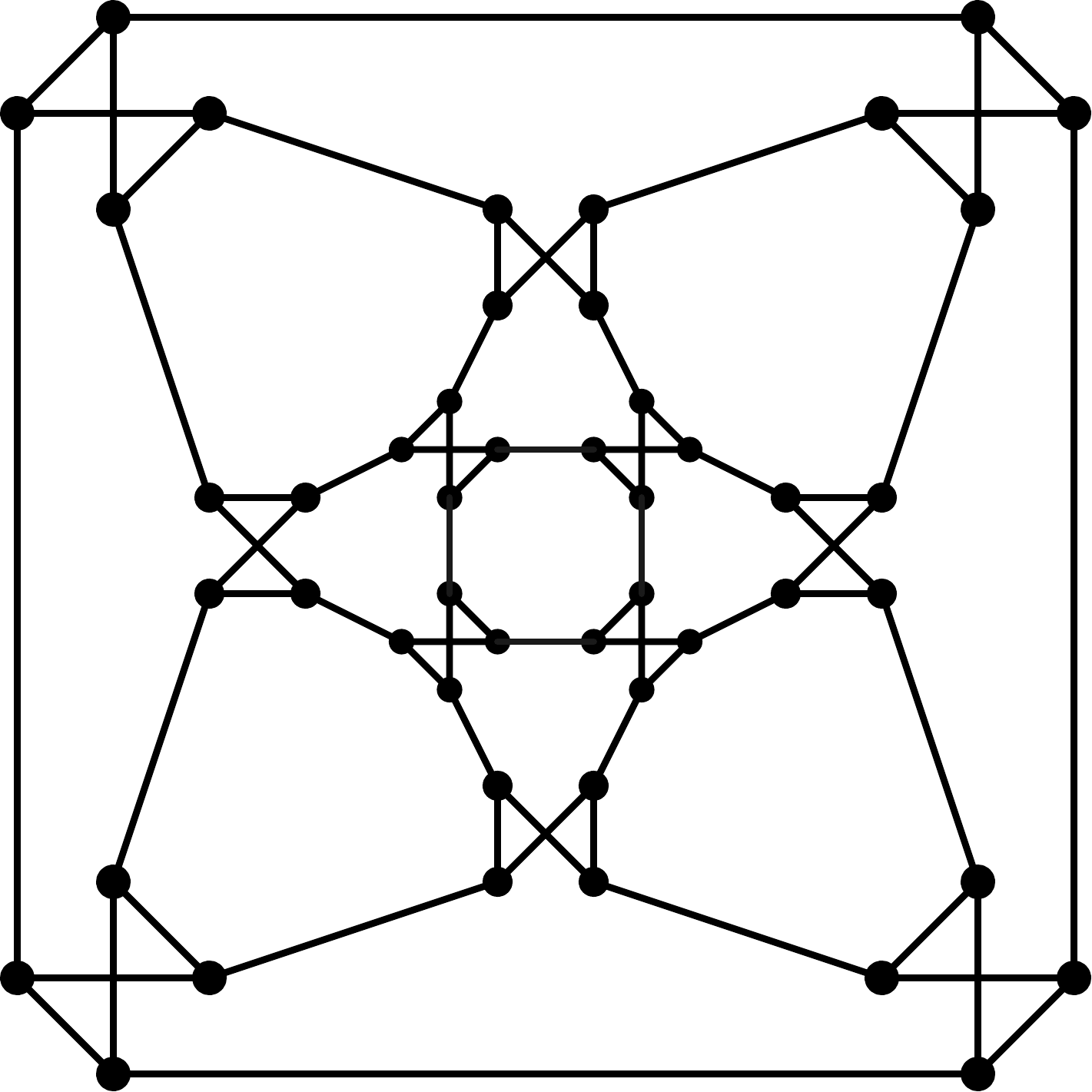}\hfill%
    \includegraphics[height=5.5\baselineskip]{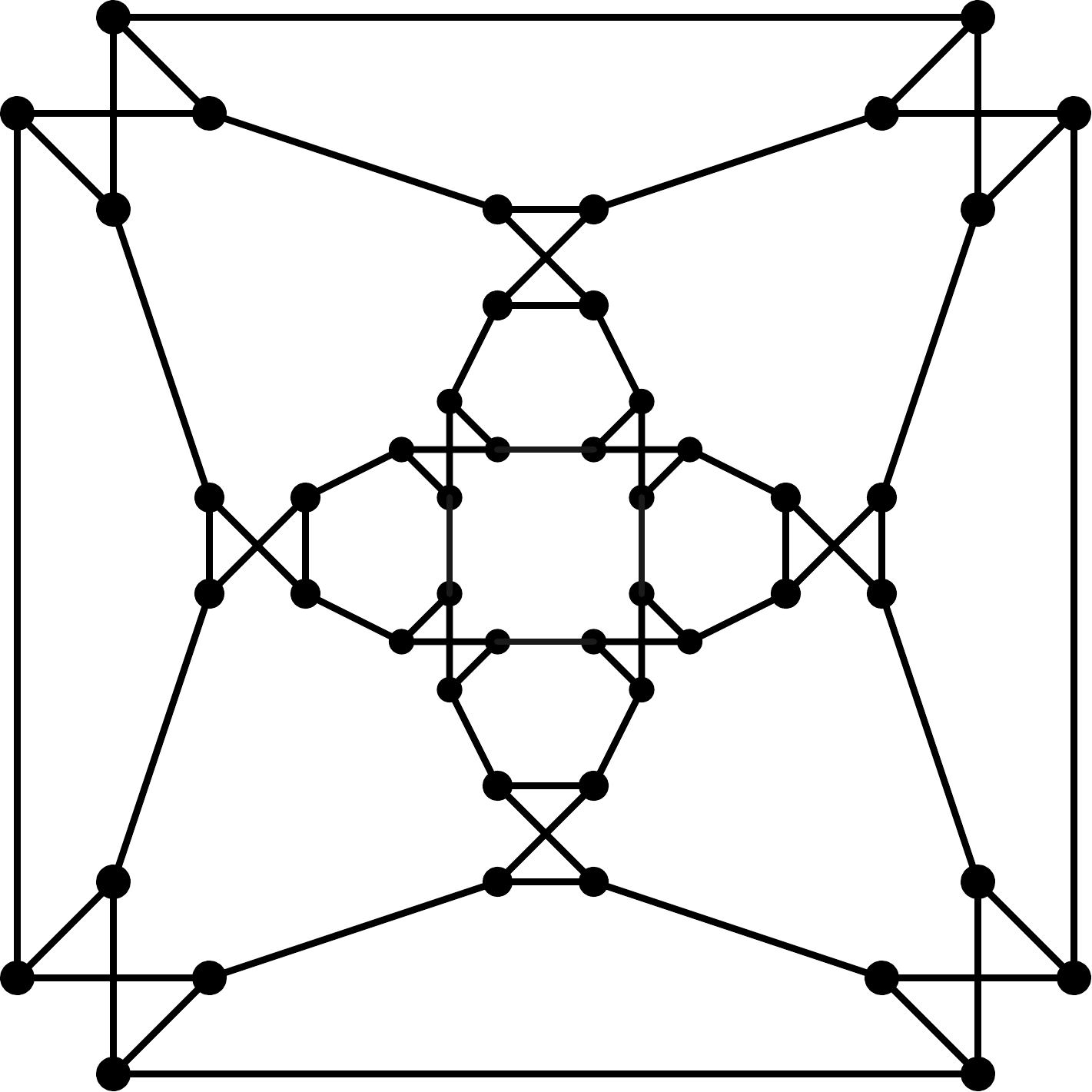}\hfill%
    \captionof{figure}{The line-graph of the cube, its $K_4$-expansion, and its vertex-transitive $C_4$-expansions}
    \label{fig:LineGraphCube}
\end{figure}

\cref{mainthm:transitive,intro:triVxCon} should be combined in practice to obtain a characterisation of all vertex-transitive finite connected graphs that are not 2-quasi-5-connected.
It turns out that the $K_3$-expansions mentioned in \cref{intro:triVxCon} can be excluded for all 3-regular outcomes of \cref{mainthm:transitive}, with only two exceptions:

\newpage

\begin{maincorollary}\label{mainCor}
    A connected finite graph~$G$ is vertex-transitive if and only if it is one of the following:
    \begin{itemize}
        \item one of the graphs in \emph{\cref{mainthm:transitive}} \cref{mainitm:transitive-1}--\cref{mainitm:transitive-7} (all outcomes possible),
        \item the $K_3$-expansion of a 3-regular 2-quasi-5-connected arc-transitive graph,
        \item the $K_3$-expansion of a cube,
        \item a cycle, $K_2$ or~$K_1$.
    \end{itemize}
\end{maincorollary}

\noindent\textbf{More related work.}
The combination of high symmetry and planarity has received a fair share of attention, see for example~\cite{PlanarCubicCayley,PlanarCayleyEnumerableOne,PlanarCayleyEnumerableTwo,PlanarTransitiveGraphs}.
Presentations of groups have been generalised beyond Cayley graphs to capture vertex-transitive graphs~\cite{AgelosVTpresentation}.
In an effort to generalise Bass-Serre theory to quasi-transitive graphs, some seminal results in the area have been generalised, see for example~\cite{AccessibilityVT,StallingsQT}.
A recent trend in the area is to study quasi-transitive graphs from the perspective of coarse graph theory, see for example~\cite{CoarseGridCounter,CoarseHalinGrid,AsymptoticGrids,daviesStringGraphs,UgoLouis,UgoNote,Ugo,QTwithoutKN}.
Coverings of graphs, in the sense of general topology, provide a particularly well-behaved class of quasi-transitive graphs.
Coverings are currently being used in graph-minor theory to study local phenomena in finite graphs through their infinite coverings, see for example~\cite{Ponyhof,canonicalGraphDec,georgakopoulos2017covers}.

\medskip

\noindent\textbf{Organisation of this paper.}
In \cref{sec:vertexTransitive}, we prove a version of \cref{mainthm:transitive} for 4-connected graphs, which is \cref{4conCase:transitive}.
In \cref{sec:3regQ4C}, we prove a version of \cref{mainthm:transitive} for 3-regular quasi-4-connected graphs, which is \cref{3regCase:transitive}, and combine it with \cref{4conCase:transitive} to prove \cref{mainthm:transitive}.
In \cref{sec:4regArcTrans}, we derive \cref{mainCor}.
For the sake of completeness, in \cref{sec:3Con} we quickly derive \cref{intro:triVxCon} from \cite[Cor.~2]{TriAiC}.\medskip

\noindent\textbf{Terminology.}
For graph-theoretic terminology we follow \cite{bibel}.
We assume familiarity with the terminology in \cite[§2]{TetraDecomp}.
Sometimes we will use results from \cite{TetraDecomp} which may rely on more terminology from~\cite{TetraDecomp}, but not much.

\section{Analysing 4-connectivity}\label{sec:vertexTransitive}

In this section, we prove \cref{mainthm:transitive}.
Recall that an \defn{$H$-expansion} of a $|V(H)|$-regular graph~$G$ is a graph obtained from~$G$ by replacing the vertices~$v$ of~$G$ with pairwise disjoint copies~$H_v$ of~$H$ and by replacing every edge $e=uv$ of~$G$ with an $H_u$--$H_v$ edge~$\hat e$ so that~$\hat e_1$ and~$\hat e_2$ share no ends for every two $e_1 \neq e_2\in E(G)$.
Thus, the edges~$\hat e$ form a perfect matching of the $H$-expansion.
Note that if~$H$ is complete then the $H$-expansion of~$G$ is unique up to isomorphism.
The \defn{clique-expansion} of a graph~$G$ is the graph~$G'$ that is obtained from~$G$ by replacing every vertex~$v \in V(G)$ by a clique~$C_v$ of size the degree of~$v$, and every edge~$e=uv$ by a $C_u$--$C_v$ edge~$\hat e$ such that these edges~$\hat e$ form a perfect matching in~$G'$.

Recall that a graph~$G$ is \defn{arc-transitive} if for every two ordered pairs~$(u_1,v_1),(u_2,v_2)$ of vertices forming edges $u_1 v_1,u_2 v_2\in E(G)$, there is an automorphism $\varphi$ of~$G$ such that~$(\varphi(u_1),\varphi(v_1))=(u_2,v_2)$.

\begin{theorem}\textnormal{\cite[Theorem~3.4.2]{Godsil}}\label{GodsilResult}
    Every finite vertex-transitive $d$-regular graph is $k$-connected for $k:=\lceil \frac{2}{3}(d+1)\rceil$.
\end{theorem}

\begin{corollary}\label{4reg4con}
    Every finite $d$-regular graph with $d\ge 4$ is 4-connected.\qed
\end{corollary}

A graph~$G$ is \defn{essentially 5-connected} if~$G$ is 4-connected and every tetra-separation $(A_1,A_2)$ of $G$ has only edges in its separator and at least one side~$A_i$ with~$G[A_i]=K_4$.

\begin{definition}\label{dfn:transitive-cycles}
Let $G$ be a graph with a cycle-decomposition $\cO=(O,(G_t)_{t \in V(O)})$.
Let $\ell$ denote the length of the cycle~$O$.
We say that~$G$ is a
    \begin{itemize}[leftmargin=*]
        \item \defn{cycle of~$K_4$-bags of length~$\ell$} if all bags of $\cO$ are $K_4$'s and all adhesion-sets of~$\cO$ are disjoint (\cref{fig:CycleOfK4s});
        \item \defn{cycle alternating between~$K_4$-bags and~$C_4$-bags of length~$\ell$} if the bags along $O$ alternate between being $K_4$'s and being 4-cycles, and all adhesion-sets of $\cO$ are disjoint (\cref{fig:CycleAlternating});
        \item \defn{cycle of triangle-bags of length~$\ell$} if all bags are triangles, and every two non-consecutive edges on~$O$ have disjoint adhesion-sets (\cref{fig:cycleOfTriangles});
        \item \defn{cycle alternating between $K_{2,2}$-bags and $C_4$-torsos of length~$\ell$} if the torsos along~$O$ alternate between being $K_4$'s and $C_4$'s, and every adhesion-set of~$\cO$ is independent and all adhesion-sets of~$\cO$ are disjoint (so the $K_4$-torsos have $K_{2,2}$-bags) (\cref{fig:CycleAltK22C4});
        \item \defn{cycle of $C_4$-bags of length~$\ell$} if all bags of~$\cO$ are $C_4$'s and all adhesion-sets of~$\cO$ induce disjoint $K_2$'s (\cref{fig:CycleC4sNormal});
        \item \defn{cycle of~$K_{2,2}$-bags of length~$\ell$} if all bags of $\cO$ are $K_{2,2}$'s, all torsos of $\cO$ are $K_4$'s and all adhesion-sets of~$\cO$ are disjoint (\cref{fig:CycleOfK22s}).
    \end{itemize}
\end{definition}

\begin{observation}{\,}
\begin{itemize}[leftmargin=*]
    \item Every cycle of~$K_4$-bags of length~$\ell$ is isomorphic to~$K_2 \boxtimes C_\ell$.
    \item Every cycle of~$K_4$-bags of length~3 is isomorphic to~$K_6$.
    \item Every cycle of~$K_{2,2}$-bags of length~3 is isomorphic to the octahedron.
    \item Every cycle of~$K_{2,2}$-bags of length~4 is isomorphic to~$K_{4,4}$.
    \item Every cycle alternating between~$K_4$-bags and~$C_4$-bags of length~$\ell\in\{4,6\}$ is essentially 5-connected.
    \item Every cycle of triangle-bags of length~$\ell$ is isomorphic to~$C_\ell^2$ (the graph obtained from the $\ell$-cycle~$C_\ell$ by joining every vertex~$v$ to all vertices of distance at most~2 from~$v$ on~$C_\ell$).
    \item Every cycle of~triangle-bags of length~5 is isomorphic to~$K_5$.
    \qed
\end{itemize}
\end{observation}

\begin{lemma}
\label{lem:transitive-totally-nested}
    Let $G$ be a vertex-transitive 4-connected finite graph, and let $(A_1,A_2)$ be a totally-nested tetra-separation in~$G$. Then $S(A_1,A_2)$ consists of four edges, and $(A_1,A_2)$ has a side~$A_i$ with $G[A_i]=K_4$.
\end{lemma}
\begin{proof}
    Let $M$ be the union of the $\Aut(G)$-orbits of $(A_1,A_2)$ and $(A_2,A_1)$. Let $(C,D)$ be $\leq$-minimal in~$M$; that is, $(C',D') \not< (C,D)$ for every $(C',D') \in M$.

    We claim that $S(C,D)$ contains no vertex.
    Assume otherwise that $u\in S(C,D)$ is a vertex.
    Then let $v$ be an arbitrary vertex in $C \sm D$ and let $\phi \in \Aut(G)$ such that $\phi(u)=v$. 
    Then $(\varphi(C),\varphi(D))<(C,D)$ or $(\varphi(D),\varphi(C))<(C,D)$, contradicting the $\leq$-minimality of~$(C,D)$.
    Hence $S(C,D)$ consists of four edges.

    We claim that every vertex in $C$ is incident to an edge in~$S(C,D)$.
    Assume for a contradiction that $v\in C$ is not.
    Let $u\in C\sm D$ be an endvertex of one of the four edges in $S(C,D)$ and let $\phi \in \Aut(G)$ such that $\varphi(u)=v$.
    Then $(\varphi(C),\varphi(D))<(C,D)$ or $(\varphi(D),\varphi(C))<(C,D)$, contradicting the $\leq$-minimality of~$(C,D)$.
    
    Finally, since $G$ is 4-connected, every vertex has degree $\ge 4$, so $G[C]$ is a~$K_4$.
    The tetra-separation $(A_1,A_2)$ shares the above properties of $(C,D)$ since $(C,D)\in M$.
\end{proof}

\begin{corollary}
\label{cor:transitive-essentially-5-connected}
    Let $G$ be a vertex-transitive 4-connected finite graph such that every tetra-separation is totally-nested. Then~$G$ is essentially 5-connected.\qed
\end{corollary}

\begin{lemma}
\label{lem:transitive-K4m}
    Let $G$ be a vertex-transitive 4-connected finite graph that has two tetra-separations $(A,B),(C,D)$ that cross with all links empty.
    Then~$G=K_{4,4}$, which is a cycle of $K_{2,2}$-bags of length~4.
\end{lemma}
\begin{proof}
    Let~$Z$ be the centre of the crossing-diagram, which consists of four vertices by the Crossing Lemma~\cite[4.1]{TetraDecomp}.
    Note that~$G-Z$ has at least four components.

    We claim that every every component of $G-Z$ has only one vertex.
    Assume not, and let $K$ be a component of $G-Z$ with $\ge 2$ vertices.
    Let~$(U_K',W_K'):=(V(K) \cup Z, V(G-K))$ and let~$(U_K,W_K)$ be the left-right-reduction of~$(U_K',W_K')$.
    By \cite[Lemma~7.2]{TetraDecomp},~$(U_K,W_K)$ is totally-nested.
    By \cref{lem:transitive-totally-nested},~$S(U_K,W_K)$ consists of four edges, and either~$G[U_K]=K_4$ or~$G[W_K]=K_4$.
    Since~$G[W_K]$ contains two vertices from distinct components of $G-Z$ with no edge in between them, we must have~$G[U_K]=K_4$.
    By \cite[Lemma~7.1]{TetraDecomp}, we have~$U_K = V(K)$ and~$W_K = V(G - K)$.
    In particular,~$G$ is 4-regular and every vertex of~$U_K$ is in a~$K_4$. 
    By vertex-transitivity, so is every vertex of~$G$.
    It follows that there exists an edge between two vertices~$z_1$ and~$z_2$ in~$Z$. 
    But now,~$z_1$ and~$z_2$ have degree~5, contradicting 4-regularity of~$G$.
    Hence every component of $G-Z$ has only one vertex.

    By 4-connectivity of~$G$, each of the $\ge 4$ vertices of $G-Z$ has neighbourhood equal to~$Z$.
    So,~$G$ is 4-regular by vertex-transitivity.
    In particular,~$G-Z$ has exactly four vertices, and~$G[Z]$ is independent.
    So,~$G=K_{4,4}$.
\end{proof}

\begin{lemma}
\label{lem:transitive-block-bagel}
    Let $G$ be a vertex-transitive 4-connected finite graph that has two tetra-separations $(A,B),(C,D)$ that cross with all links of size~one. 
    Let~$\cO$ be the block-bagel induced by~$(A,B)$ and~$(C,D)$. 
    Then every bag of~$\cO$ is bad.
\end{lemma}
\begin{proof}
    Assume for a contradiction that $\cO=:(O,\cG)$ has a good bag~$G_t$.
    Then we have the Block-bagel bag setting~\cite[8.14]{TetraDecomp} with $t=t(0)$, which we assume in the following.
    By \cite[Key Lemma~8.27]{TetraDecomp}, $(U,W)$ is a totally-nested tetra-separation of~$G$.
    By \cref{lem:transitive-totally-nested}, the separator $S(U,W)$ consists of four edges, and either~$G[U]=K_4$ or~$G[W]=K_4$.
    It follows with vertex-transitivity that $G$ is 4-regular and that the neighbourhood around each vertex induces $K_1\sqcup K_3$ in~$G$.

    By the Helpful Lemma~\cite[8.6]{TetraDecomp}, $O$ has at least four helpful edges or vertices.
    By 4-regularity, there are exactly four helpful edges or nodes of~$O$.
    For each helpful edge, we colour its adhesion-vertex red.
    For each helpful node~$t$, there exists a unique interior vertex of $G_t$ that is a neighbour of $z'_1$ (by 4-regularity), and we colour this vertex red.
    Then exactly four vertices have been coloured red.
    By the above, the red vertices induced a $K_1\sqcup K_3$ in~$G-Z'$, where $Z'$ is the centre of the crossing-diagram.
    None of the red vertices in the $K_3$ can be an interior vertex of a bag, so all three are adhesion-vertices.
    But then $O$ would have to be a triangle, contradicting that $O$ has four helpful edges or nodes.
\end{proof}

\begin{lemma}
\label{lem:transitive-no-K3}
    Let $G$ be a vertex-transitive 4-connected finite graph that has two tetra-separations $(A,B),(C,D)$ that cross with all links of size~one. 
    Let $\cO$ be the Block-bagel induced by $(A,B)$ and $(C,D)$. 
    Then every bag of~$\cO$ is a~$K_2$.
\end{lemma}
\begin{proof}
    Write $\cO =: (O, (G_t)_{t \in V(O)})$, and let $u,v$ be the two vertices in the centre.
    By \cref{lem:transitive-block-bagel}, every bag of~$\cO$ is bad.
    In particular, every bag of~$\cO$ is a $K_2$ or a triangle.
    Assume for a contradiction that $\cO$ has a bag that is a triangle.
    Since every bag of $\cO$ is bad, it follows that all bags of $\cO$ are triangles.

    Let $G_t$ be an arbitrary bag of $\cO$, and let $x$ be the vertex in~$G_t$ that is not an adhesion-vertex. 
    By 4-connectivity of~$G$, we have~$x u, x v \in E(G)$ and~$x$ has degree~4.
    Hence~$G$ is 4-regular, and since $G_t$ was arbitrary, we get $|O|\le 4$.
    As every node of $O$ is helpful, we get $|O|\ge 4$ from the Helpful Lemma~\cite[8.6]{TetraDecomp}, so $O$ has length~4.
    Hence each of $u,v$ has four neighbours, one in the interior of each bag, so by 4-regularity the edge $uv$ cannot exist in~$G$.
    Therefore, $u$ is not contained in a triangle.
    But all bags are triangles, contradicting vertex-transitivity.
\end{proof}

\begin{lemma}
\label{lem:transitive-double-wheel}
    Let $G$ be a vertex-transitive 4-connected finite graph that has two tetra-separations $(A,B),(C,D)$ that cross with all links of size one. Then~$G$ is an octahedron, which is a cycle of $K_{2,2}$-bags of length~3.
\end{lemma}
\begin{proof}
    By \cref{lem:transitive-no-K3}, every bag of the block-bagel $\cO=(O,\cG)$ induced by $(A,B)$ and $(C,D)$ is a~$K_2$.
    Hence every edge of $O$ is helpful, so $|O|\ge 4$ by the Helpful Lemma~\cite[8.6]{TetraDecomp}. 
    Since each vertex of $G$ contained in $\bigcup_{t \in V(O)} G_t$ has degree $\le 4$ in~$G$, the graph~$G$ is 4-regular. 
    Let $u,v$ be the vertices in the centre. 
    By~\cite[Lemma~8.7]{TetraDecomp} it follows that $|O|=4$ and $uv \notin E(G)$. 
    Then~$G$ is an octahedron.
\end{proof}

\begin{lemma}
\label{lem:transitive-tutte-good-bag}
    Let $G$ be a vertex-transitive 4-connected finite graph that has two tetra-separations $(A,B),(C,D)$ that cross with all links of size two. 
    Let~$\cO$ be the Tutte-bagel induced by~$(A,B)$ and~$(C,D)$.
    Let $H_t$ be a good torso of $\cO$.
    Then the bag $G_t$ is a $K_4$ and its two neighbouring torsos are 4-cycles.
\end{lemma}
\begin{proof}
    We may assume the Tutte-bagel torso setting~\cite[9.13]{TetraDecomp} with $t=t(0)$.
    By~~\cite[Key Lemma~9.41]{TetraDecomp},~$(U,W)$ is a totally-nested tetra-separation.
    By \cref{lem:transitive-totally-nested}, $S(U,W)$ consists of four edges, and we have either~$G[U]=K_4$ or~$G[W]=K_4$.

    We claim that $G[U]= K_4$.
    Assume otherwise, so $G[W]=K_4$.
    Then $G[W]=G_s$ for some bag $G_s\neq G_t$.
    Hence $O$ has length $\le 4$, two 3-connected bags (namely $G_s,G_t$), and all other bags have triangles or 4-cycles as torsos by~\cite[Lemma~9.14]{TetraDecomp}.
    But then $\alpha(\cO)<4$ contradicts~\cite[Lemma~9.51]{TetraDecomp}.

    Hence $G[U]=K_4$.
    Next, we claim that $S(U,W)$ contains no inward edge, and assume for a contradiction that $x_1\in S(U,W)$ is an inward edge with endvertex $u$ in the interior of~$G_t$.
    By~\cite[Lemma~9.33]{TetraDecomp}, some $y'_i$, say $y'_1$, is a neighbour of $x'_1$ and $y^W_1=y'_1$.
    In particular, $y_1$ cannot be an outward edge.
    Since $S(U,W)$ consists of four edges and $y_1$ is one of them, $y_1$ must be an inward edge.
    Let $u$ be the endvertex of $y_1$ in the interior of~$G_t$.
    Then $u\neq v$ by the matching-condition for $(U,W)$.
    Let $w$ be a neighbour of $x'_1$ in $G$ but outside of~$G_t$.
    Then the vertex-set $\{u,v,w\}$ is independent in~$G$ and consists entirely of neighbours of $x'_1$.
    Since $G[U]=K_4$ and $S(U,W)$ consists of four edges, $G$ is 4-regular and every vertex lies in a $K_4$ by vertex-transitivity.
    But $\{u,v,w\}$ witnesses that $x'_1$ does not lie in a~$K_4$, a contradiction.
    Therefore, $S(U,W)$ contains no inward edge.

    So all edges in $S(U,W)$ are outward edges.
    Hence $G_t=K_4$ and both neighbouring torsos are 4-cycles.
\end{proof}

\begin{lemma}
\label{lem:transitive-tutte-bagel}
    Let $G$ be a vertex-transitive 4-connected finite graph that has two tetra-separations $(A,B),(C,D)$ that cross with all links of size~two. 
    Let~$\cO$ be the Tutte-bagel induced by~$(A,B)$ and~$(C,D)$.
    If some torso of $\cO$ is good, then~$G$ is a cycle alternating between $K_4$-bags and $C_4$-bags of length~$\ge 8$.
\end{lemma}
\begin{proof}
    Write $\cO =: (O,(G_t)_{t \in V(O)})$.
    Write $O=t(0)\,t(1)\ldots t(m)\, t(0)$ with indices in~$\Z_{m+1}$.
    By \cref{lem:transitive-tutte-good-bag}, if a torso such as $H_{t(0)}$ is good, then its bag is a~$K_4$ and the two neighbouring torsos $H_{t(-1)}$ and $H_{t(1)}$ are 4-cycles.
    Then, by~\cite[Lemma~9.4]{TetraDecomp}, the torsos $H_{t(-2)}$ and $H_{t(2)}$ are 3-connected.
    Hence, both of these torsos are good as they neighbour torsos that are 4-cycles.
    By induction, we get that~$O$ has even length, that bags $G_{t(i)}$ are $K_4$'s for even~$i$, and that torsos $H_{t(i)}$ are 4-cycles for odd~$i$.
    By~\cite[Lemma~9.51]{TetraDecomp}, there are $\ge 4$ bags that are~$K_4$'s.
    It follows that~$G$ is a cycle alternating between $K_4$-bags and $C_4$-bags of length~$\ge 8$.
\end{proof}

\begin{lemma}
\label{lem:tutte-bagel-wheel}
    Let $\cO$ be a Tutte-bagel of a 4-connected graph.
    Let $G_t$ be a bag of $\cO$ on five vertices.
    Let $\pi,\pi'$ be the two adhesion-sets at~$G_t$, and let $v$ be the interior vertex of~$G_t$.
    Then the bag $G_t$ contains a full matching between $\pi$ and~$\pi'$, and $v$ sends edges to all four vertices in $\pi\cup\pi'$.
\end{lemma}
\begin{proof}
    Since $G$ is 4-connected, $v$ sends edges to all four vertices in $\pi\cup\pi'$.
    Since the torso $H_t$ is 3-connected with $|H_t|=5$, it is either a $4$-wheel, a $K_5^-$ or a $K_5$.
    In each case, $H_t$ contains a full matching between $\pi$ and $\pi'$, and both edges in the matching also exist in the bag~$G_t$.
\end{proof}

\begin{lemma}
\label{lem:transitive-cube}
    Let $G$ be a vertex-transitive 4-connected finite graph that has two tetra-separations $(A,B),(C,D)$ that cross with all links of size two. 
    Let~$\cO$ be the Tutte-bagel induced by~$(A,B)$ and~$(C,D)$.
    Assume that $\cO$ has a torso on five vertices.
    Then~$G$ is the line graph of the cube.
\end{lemma}
\begin{proof}
    Write $\cO=:(O,\cG)$.
    By \cref{lem:transitive-tutte-bagel}, all bags of $\cO$ are bad.
    Let $H_t$ be a torso of $\cO$ on five vertices.
    Since the unique interior vertex~$x$ of~$H_t$ has degree~4, the graph~$G$ is 4-regular.
    Since~$H_t$ is bad, it is neighboured by a 3-connected torso~$H_s$.
    Since $H_s$ is bad as well, it has $\le 5$ vertices.
    That is, the torso of~$H_s$ either is a~$K_4$ or has exactly five vertices.

    We claim that $H_s$ has exactly five vertices.
    Otherwise $H_s=K_4$.
    Let $\{u,v\}$ denote the adhesion-set in between the torsos $H_t$ and~$H_s$.
    Then $uv \notin E(G)$ by 4-regularity of~$G$.
    Let $\{u',v'\}$ denote the other adhesion-set at~$H_s$.
    Since~$v$ is contained in at least two triangles in~$G$, we must have $u'v' \in E(G)$.
    The vertex $v$ is contained in exactly two triangles: one from $H_s$ and the other is $vu'v'$.
    The vertex~$v'$, however, is contained in at least two triangles (namely $vu'v'$ and $uu'v'$), two of which intersect in the edge~$u'v'$.
    This contradicts the vertex-transitivity of~$G$.
    Hence $H_s$ has exactly five vertices.

    It follows by induction that all torsos of $\cO$ have exactly five vertices.
    By \cref{lem:tutte-bagel-wheel}, each bag contains a 2-matching between its two adhesion-sets, and the interior vertex of the bag sends edges to all four other vertices in the bag.
    By 4-regularity and since every bag looks like this, these are all the edges that there are in the bags.

    Since each bag has a unique interior vertex, the set $X$ of all these vertices has size~$|O|$.
    Since all torsos are 3-connected, we have $|O|\ge 4$ by~\cite[Lemma~9.51]{TetraDecomp}.
    If~$|O| \geq 5$, then every vertex in $X$ is contained in exactly two induced 4-cycles in~$G$, while every vertex in an adhesion set is contained in exactly one induced 4-cycle in~$G$, contradicting vertex-transitivity. 
    Therefore, we have~$|O|=4$. 
    Consider the graph~$H :=  G - X$, which is either an 8-cycle or a disjoint union of two 4-cycles. 
    If~$H$ is an 8-cycle, then every vertex in~$X$ is contained in two induced 4-cycles of~$G$ while every vertex in $V(G - X)$ is only contained in one induced 4-cycle of~$G$, contradicting vertex-transitivity. 
    Therefore $H$ is a disjoint union of two 4-cycles, and it follows that~$G$ is the line graph of the cube.
\end{proof}

\begin{lemma}
\label{lem:transitive-bagel-without-K5}
    Let $G$ be a vertex-transitive 4-connected finite graph that has two tetra-separations $(A,B),(C,D)$ that cross with all links of size~two.
    Let~$\cO$ be the Tutte-bagel induced by~$(A,B)$ and~$(C,D)$.
    Assume that every bag of $\cO$ has $\le 4$ vertices.
    Then~$G$ is one of the following graphs:
    \begin{itemize}
        \item a cycle alternating between $K_4$-bags and $C_4$-bags of length~$\ge 8$;
        \item a cycle of triangle-bags of length~$\ge 6$;
        \item a cycle of~$K_{2,2}$-bags of length~$\ge 4$;
        \item a cycle of~$K_4$-bags of length~$\ge 4$.
    \end{itemize}
\end{lemma}
\begin{proof}
    If~$\cO$ has a good torso, then~$G$ is a cycle alternating between $K_4$-bags and $C_4$-bags by \cref{lem:transitive-tutte-bagel}, and it is of length $\ge 8$ by~\cite[Lemma~9.51]{TetraDecomp}.
    Hence, we may assume that all torsos of $\cO$ are bad.

    We claim that $\cO$ has no 4-cycle as torso.
    Indeed, if $\cO$ has a 4-cycle as a torso, then the neighbouring torsos are 3-connected by~\cite[Lemma~9.4]{TetraDecomp}, and they are good, a contradiction.
    Thus,
    \begin{equation}\label{lem:transitive-bagel-without-K5-eq}
        \textit{all torsos are either triangles or $K_4$'s.}
    \end{equation}

    We claim that if some torso of $\cO$ is a triangle, then also its neighbouring torsos are triangles.
    For this, let $H_t$ and $H_s$ be neighbouring torsos with $H_t$ a triangle.
    By \cref{lem:transitive-bagel-without-K5-eq}, the torso $H_s$ is either a triangle or a~$K_4$.
    Assume for a contradiction that $H_s=K_4$.
    Let $\{u,v\}$ be the adhesion-set in between $H_t$ and $H_s$, so that $v$ is the tip of the triangle-torso~$H_t$, say.
    Since every vertex has degree~$\ge 4$, the edge $uv$ must exist in $G$, and $G$ must be 4-regular.
    Let $w$ be the vertex of the triangle $H_t$ besides $u$ and~$v$.
    Let $G_r$ be the bag neighbouring $G_t$ besides~$G_s$.
    Then $v$ has exactly one neighbour in the bag~$G_r$.
    By \cref{lem:transitive-bagel-without-K5-eq}, the torso $H_r$ can only be a triangle or a~$K_4$.
    As $v$ has exactly one neighbour in the bag~$G_r$, the torso $H_r$ cannot be a~$K_4$, so it must be a triangle.
    But then the edge $vw$ exists by~\cite[Lemma~9.4]{TetraDecomp}, so $v$ has two neighbours in~$G_r$, giving it degree five in~$G$, contradicting 4-regularity.
    Therefore, every triangle-torso of $\cO$ is neighboured by two triangle-torsos.

    So if $\cO$ has at least one triangle-torsos, then all torsos are triangles, and all adhesion-sets are spanned by edges by~\cite[Lemma~9.4]{TetraDecomp}.
    Hence $G$ is a cycle of triangle-bags, and it has length $\ge 6$ by~\cite[Lemma~9.51]{TetraDecomp} and since $G \neq K_5$.

    So we may assume that $\cO$ has no triangle-torsos.
    Then by \cref{lem:transitive-bagel-without-K5-eq}, all torsos of $\cO$ are $K_4$'s.
    Then either all adhesion-sets are spanned by edges in~$G$, in which case~$G$ is a cycle of~$K_4$-bags; or none of them are, in which case~$G$ is a cycle of~$K_{2,2}$-bags.
    In both cases, the length is $\ge 4$ by~\cite[Lemma~9.51]{TetraDecomp}.
\end{proof}

\begin{lemma}
\label{lem:transitive-bagel-case}
    Let $G$ be a vertex-transitive 4-connected finite graph that has two tetra-separations $(A,B),(C,D)$ that cross with all links of size~two. Then~$G$ is one of the following:
    \begin{itemize}
        \item a cycle alternating between $K_4$-bags and $C_4$-bags of length~$\ge 8$;
        \item a cycle of triangle-bags of length~$\ge 6$;
        \item a cycle of~$K_{2,2}$-bags of length~$\ge 4$;
        \item a cycle of~$K_4$-bags of length~$\ge 4$;
        \item line graph of the cube.
    \end{itemize}
\end{lemma}
\begin{proof}
    We combine~\cref{lem:transitive-cube,lem:transitive-bagel-without-K5}.
\end{proof}

\begin{corollary}
\label{cor:transitive-crossing}
    Let $G$ be a vertex-transitive 4-connected finite graph that has two tetra-separations $(A,B),(C,D)$ that cross.
    Then~$G$ is one of the following:
    \begin{itemize}
        \item a cycle alternating between $K_4$-bags and $C_4$-bags of length~$\ge 8$;
        \item a cycle of triangle-bags of length~$\ge 6$;
        \item a cycle of~$K_{2,2}$-bags of length~$\ge 3$;
        \item a cycle of~$K_4$-bags of length~$\ge 4$;
        \item line graph of the cube.
    \end{itemize}
\end{corollary}
\begin{proof}
    By the Crossing Lemma~\cite[4.1]{TetraDecomp}, all links have the same size $0$, $1$ or~$2$.
    If all links have size~$0$, then $G$ is a cycle of~$K_{2,2}$-bags of length~$4$ by \cref{lem:transitive-K4m}.
    If all links have size~$1$, then $G$ is a cycle of~$K_{2,2}$-bags of length~$3$ by \cref{lem:transitive-double-wheel}.
    Otherwise all links have size~$2$, and we apply \cref{lem:transitive-bagel-case}.
\end{proof}

\begin{observation}
\label{obs:expansion-maintains-degrees}
    Let~$G'$ be the clique-expansion of a graph~$G$.
    Let~$v$ be a vertex of degree~$d$ in~$G$, and let~$C_v$ be the clique of~$G'$ that corresponds to~$v$.
    Then every vertex in~$C_v$ has degree~$d$ in~$G'$.\qed
\end{observation}

\begin{observation}
    \label{obs:expansion-clique-partition}
    Let~$G$ be a graph of minimum degree at least~$2$, and let~$G'$ be the clique-expansion of~$G$.
    Then there is a unique partition~$\cP$ of~$V(G')$ into (the vertex-sets of) $|\cP|=|V(G)|$ cliques.
    Moreover, the contraction-minor $G' / \cP$ is isomorphic to~$G$.\qed
\end{observation}

\begin{observation}
\label{obs:expansion-isomorphic}
    The clique-expansion of a graph~$G$ is isomorphic to the graph~$H$ with 
    \begin{align*}
        V(H) & {}=\{(u,v),(v,u) : uv \in E(G)\}\text{ and}\\
        E(H) & {}=\Big\{\{(u,v),(v,u)\} : uv \in E(G)\Big\}\,\cup\,\Big\{\{(u,v_1),(u,v_2)\} : uv_1,uv_2 \in E(G) \Big\}.
    \end{align*}\qed
\end{observation}

\begin{lemma}
\label{lem:expansion-edge-transitive}
    Let~$G$ be an $r$-regular graph with $r \geq 3$.
    Then the following assertions are equivalent:
    \begin{enumerate}
        \item\label{lem:expansion-edge-transitive1} the $K_r$-expansion of~$G$ is vertex-transitive;
        \item\label{lem:expansion-edge-transitive2} $G$ is arc-transitive.
    \end{enumerate}
\end{lemma}
\begin{proof}
    Let $G'$ denote the $K_r$-expansion of~$G$.
    By \cref{obs:expansion-clique-partition},~$G'$ has a unique partition~$\cP$ of its vertex-set into vertex-sets of $r$-cliques.
    Let $\gamma\colon \cP \to V(G)$ be the corresponding bijection.
    Since $\cP$ is unique, $\Aut(G')$ acts on~$\cP$.
    
    \cref{lem:expansion-edge-transitive1} $\Rightarrow$ \cref{lem:expansion-edge-transitive2}. 
    Let $(u_1,v_1),(u_2,v_2)$ be ordered pairs of vertices in~$G$ with $u_i v_i\in E(G)$.
    We will find an automorphism~$\phi$ of~$G$ such that $\varphi(u_1)=u_2$ and $\varphi(v_1)=v_2$.
    For this, we abbreviate $x_i :=(u_i,v_i)$.
    Let~$P_i \in \cP$ be the unique partition-class of~$\cP$ that contains~$x_i$. 
    Since~$G'$ is vertex-transitive by assumption, there is an automorphism~$\phi'$ of~$G'$ that maps~$x_1$ to~$x_2$.
    Note that $\varphi$ then maps $P_1$ to~$P_2$.
    We define $\phi := \gamma \circ \phi' \circ \gamma^{-1}$ by abuse of notation, using that $\Aut(G')$ acts on~$\cP$.
    Clearly, $\phi$ is a bijection $V(G)\to V(G)$.
    Moreover, we have
    \begin{align*}
        & \text{ }xy\in E(G)\\
        \Leftrightarrow & \text{ there is a $\gamma^{-1}(x)$--$\gamma^{-1}(y)$ edge in~$G'$}\\
        \Leftrightarrow & \text{ there is a $\phi'(\gamma^{-1}(x))$--$\phi'(\gamma^{-1}(y))$ edge in~$G'$}\\
        \Leftrightarrow & \text{ $\phi(x)\phi(y) \in E(G)$.}
    \end{align*}
    So,~$\phi$ is an automorphism of~$G$.
    Moreover, we have
    \[
        \phi(u_1) = \gamma(\phi'(\gamma^{-1}(u_1))) = \gamma(\phi'(P_1)) = \gamma(P_2) = u_2.
    \]
    Similarly, $\phi(v_1)=v_2$.

    \cref{lem:expansion-edge-transitive2} $\Rightarrow$ \cref{lem:expansion-edge-transitive1}.
    We assume without loss of generality that $G'$ takes the form described in~\cref{obs:expansion-isomorphic}.
    Let $x_1,x_2 \in V(G')$ be two vertices.
    We will find an automorphism~$\phi'$ of~$G'$ that maps~$x_1$ to~$x_2$.
    We write $x_i=:(u_i,v_i)$ using \cref{obs:expansion-isomorphic}, so $u_i v_i$ is an edge of $G$.
    Since~$G$ is arc-transitive, there is an automorphism~$\phi$ of~$G$ such that $\varphi(u_1)=u_2$ and $\varphi(v_1)=v_2$.
    Now $\varphi$ defines an automorphism $\varphi'$ of $G'$ that sends every vertex $x=(u,v)$ to the vertex $\varphi'(x):=(\varphi(u),\varphi(v))$.
    The automorphism $\varphi'$ satisfies
    \[
        \phi'(x_1)=(\phi(u_1),\phi(v_1)) = (u_2,v_2) = x_2.\qedhere
    \]
\end{proof}

\begin{corollary}
    \label{cor:expansion-mixed-separation-not-transitive}
    Let~$G$ be an $r$-regular graph with $r\ge 3$.
    Let $(A,B)$ be a mixed-$r$-separation of $G$ such that $S(A,B)$ consists only of edges and~$S(A,B)$ is a matching and $G[A] = K_r$.
    Then~$G$ is not arc-transitive.
    Moreover, the $K_r$-expansion of~$G$ is not vertex-transitive.
\end{corollary}
\begin{proof}    
    By \cref{lem:expansion-edge-transitive}, it suffices to show that~$G$ is not arc-transitive.
    Indeed, if~$u$ is a vertex in~$A$, then $G[N(u)]=K_1 \sqcup K_{r-1}$ with $r-1 \geq 2$.
    Let~$v_1$ be the isolated vertex in~$G[N(u)]$ and let~$v_2 \neq v_1$ be another vertex in~$N(u)$.
    Then there is no automorphism of~$G$ that maps~$(u,v_1)$ to~$(u,v_2)$.
\end{proof}

\begin{corollary}
    \label{cor:iterated-expansion-not-transitive}
    Let $r \geq 3$.
    The $K_r$-expansion of an $r$-regular graph~$G$ is not arc-transitive.
    In particular, the $K_r$-expansion of the $K_r$-expansion of an $r$-regular graph~$G$ is not vertex-transitive.
\end{corollary}
\begin{proof}
    Follows from \cref{cor:expansion-mixed-separation-not-transitive}.
\end{proof}

\begin{corollary}
    \label{cor:cycle-of-triangles-not-transitive}
    A cycle of triangle-bags of length~$\geq 6$ is not arc-transitive.
    In particular, the $K_4$-expansion of a cycle of triangle-bags of length~$\ge 6$ is not vertex-transitive.
\end{corollary}
\begin{proof}
    Let~$G$ be a cycle of triangle-bags of length~$\ge 6$, and let~$\cO$ be the corresponding cycle-decomposition.
    Note that every free edge of~$\cO$ is contained in exactly one triangle of~$G$, while every other edge of~$G$ is contained in exactly two triangles.
    So,~$G$ is not edge-transitive.
    In particular, it is not arc-transitive.
    By \cref{lem:expansion-edge-transitive}, the $K_4$-expansion of a cycle of triangle-bags of length~$\ge 6$ is not vertex-transitive. 
\end{proof}

\begin{corollary}
    \label{cor:symmetric-crossing}
    Let $G$ be a finite graph that is arc-transitive, 4-connected, 4-regular and has two tetra-separations that cross.
    Then~$G$ is the line graph of the cube or a cycle of~$K_{2,2}$-bags of length~$\ge 3$.
\end{corollary}
\begin{proof}
    This follows from \cref{cor:transitive-crossing} together with the following observations.
    First, a cycle of~$K_4$-bags of length~$\ge 4$ is not 4-regular.
    Second, a cycle alternating between $K_4$-bags and $C_4$-bags of length~$\ge 8$ is not arc-transitive by \cref{cor:expansion-mixed-separation-not-transitive}.
    Finally, a cycle of triangle-bags of length~$\ge 6$ is not arc-transitive by \cref{cor:cycle-of-triangles-not-transitive}.
\end{proof}

\begin{lemma}
\label{lem:transitive-partition-into-K4}
    Let~$G$ be a vertex-transitive essentially 5-connected finite graph. 
    Then either~$G$ is quasi-5-connected or there is a partition~$\cP$ of~$V(G)$ into vertex-sets of~$K_4$'s such that the edges in $G$ with ends in different classes of~$\cP$ form a perfect matching in~$G$.
\end{lemma}
\begin{proof}
    If~$G$ does not have a tetra-separation, then~$G$ is quasi-5-connected by \cite[Lemma~2.3]{TetraDecomp}.
    Hence, we may assume that~$G$ has a tetra-separation~$(A,B)$.
    Since~$G$ is essentially 5-connected, the separator of $(A,B)$ consists of four edges and there is a side, say~$A$, that induces a~$K_4$ in~$G$.
    Let~$u \in A$ and observe that $G[N(u)]=K_1 \sqcup K_3$; in particular,~$u$ is contained in exactly one~$K_4$.
    By vertex-transitivity, so is every vertex of~$G$.
    It follows that~$G$ admits a partition of its vertex-set into the vertex-sets of~$K_4$'s.
    Moreover,~$G$ is 4-regular since~$u$ has degree~4 in~$G$. Then it follows that the edges between distinct partition-classes of~$\cP$ form a perfect matching in~$G$.
\end{proof}

\begin{lemma}
\label{lem:transitive-contract-essentially-5-connected}
    Let~$G$ be a vertex-transitive essentially 5-connected finite graph. Then~$G$ is one of the following:
    \begin{itemize}
        \item quasi-5-connected;
        \item a cycle alternating between $K_4$-bags and $C_4$-bags of length $\ell \in \{4,6\}$;
        \item the $K_4$-expansion of an arc-transitive 4-connected 4-regular graph.
    \end{itemize}
\end{lemma}
\begin{proof}
    Assume that~$G$ is not quasi-5-connected.
    Then by \cref{lem:transitive-partition-into-K4}, there is a partition~$\cP$ of~$V(G)$ into the vertex-sets of~$K_4$'s such that edges between distinct partition-classes of~$\cP$ form a perfect matching in~$G$.
    We distinguish the following cases.

    \casen{Case 1:} there are two parts~$P_1 \neq P_2$ of~$\cP$ such that there are at least three $P_1$--$P_2$ edges~$e_1,e_2,e_3$ in~$G$.
    Let~$f_i \notin \{e_1,e_2,e_3\}$ be the remaining edge incident with a vertex in~$P_i$.
    By 4-connectivity of~$G$, $\{f_1,f_2\}$ is not a mixed-2-separator. 
    It follows that $f_1 = f_2$, so~$G$ is a cycle alternating between $K_4$-bags and $C_4$-bags of length~4.

    \casen{Case 2:} there are two parts~$P_1 \neq P_2$ of~$\cP$ such that there are exactly two $P_1$--$P_2$ edges in~$G$. 
    Let~$e_1,e_2$ be the two edges incident with~$P_1$ but not with~$P_2$, and let~$f_1,f_2$ be the two edges incident with~$P_2$ but not with~$P_1$.
    We define $(A,B):=(P_1 \cup P_2, V(G) \sm (P_1 \cup P_2))$.
    Then $S(A,B)=\{e_1,e_2,f_1,f_2\}$, so~$(A,B)$ is a tetra-separation.
    By assumption, there is a side of~$(A,B)$ that induces a~$K_4$.
    Since $|A|=8$, we must have $G[B] = K_4$.
    It follows that~$G$ is a cycle alternating between $K_4$-bags and $C_4$-bags of length~$6$.

    \casen{Case 3:} for every two parts~$P_1 \neq P_2$ of~$\cP$, there is at most one $P_1$--$P_2$ edge in~$G$. Then we define $H:=G / \cP$, which is a 4-regular graph.
    Moreover, observe that~$G$ is the $K_4$-expansion of~$H$.
    Then by \cref{lem:expansion-edge-transitive},~$H$ is arc-transitive. By \cref{4reg4con},~$H$ is 4-connected.
\end{proof}

\begin{corollary}
\label{cor:symmetric-contract-essentially-5-connected}
    Let~$G$ be an arc-transitive essentially 5-connected graph. Then~$G$ is quasi-5-connected.
\end{corollary}
\begin{proof}
    Follows from \cref{lem:transitive-contract-essentially-5-connected} together with the following observations.
    First, a cycle alternating between $K_4$-bags and $C_4$-bags of length $\ell \in \{4,6\}$ is not arc-transitive by \cref{cor:expansion-mixed-separation-not-transitive}.
    Second, the $K_4$-expansion of a graph is not arc-transitive by \cref{cor:iterated-expansion-not-transitive}.
\end{proof}

\begin{proposition}
    \label{4conCase:transitive}
    A 4-connected finite graph $G$ is vertex-transitive if and only if $G$ either is
    \begin{enumerate}
        \item\label{4conCase:transitive-1} a quasi-5-connected vertex-transitive graph,
        \item\label{4conCase:transitive-2} the $K_4$-expansion of a quasi-5-connected arc-transitive graph,
        \item\label{4conCase:transitive-3} a cycle of $K_4$-bags of length $\ge 4$ (\cref{fig:CycleOfK4s}),
        \item\label{4conCase:transitive-4} a cycle alternating between $K_4$-bags and $C_4$-bags of length $\ge 4$ (\cref{fig:CycleAlternating}),
        \item\label{4conCase:transitive-5} a cycle of triangle-bags of length $\ge 6$ (\cref{fig:cycleOfTriangles}),
    \end{enumerate}
    \noindent or $G$ is one of the following graphs or a $K_4$-expansion thereof:
    \begin{enumerate}[resume]
        \item\label{4conCase:transitive-6} a cycle of $K_{2,2}$-bags of length $\ge 3$ (\cref{fig:CycleOfK22s}),
        \item\label{4conCase:transitive-7} the line-graph of the cube (\cref{fig:LineGraphCube}).
    \end{enumerate}
\end{proposition}

\begin{proof}
    Let $G$ be a 4-connected finite graph.

    ($\Rightarrow$).
    Assume that $G$ is vertex-transitive.
    If~$G$ has two tetra-separations that cross, then we are done by \cref{cor:transitive-crossing}.
    Otherwise, every tetra-separation of~$G$ is totally-nested.
    Then~$G$ is essentially 5-connected by \cref{cor:transitive-essentially-5-connected}.
    By \cref{lem:transitive-contract-essentially-5-connected},~$G$ is either quasi-5-connected, or a cycle alternating between $K_4$-bags and $C_4$-bags of length $\ell \in\{4,6\}$, or the $K_4$-expansion of an arc-transitive 4-connected 4-regular graph.
    In the first two cases, we are done.
    So assume that~$G$ is the $K_4$-expansion of an arc-transitive 4-connected 4-regular graph~$H$.
    In particular,~$H$ is vertex-transitive.

    If~$H$ has two tetra-separations that cross, then we apply \cref{cor:symmetric-crossing} to get that~$H$ is the line graph of the cube or that~$H$ is a cycle of~$K_{2,2}$-bags of length~$\geq 3$. 
    Here we are done.
    Otherwise, every tetra-separation of~$H$ is totally-nested. 
    By \cref{cor:transitive-essentially-5-connected},~$H$ is essentially 5-connected.
    By \cref{cor:symmetric-contract-essentially-5-connected},~$H$ quasi 5-connected. 
    Thus, $G$ is the $K_4$-expansion of an arc-transitive quasi-5-connected graph, and we are done.

    ($\Leftarrow$). 
    Clearly, if~$G$ is one of the graphs in~\ref{4conCase:transitive-1}, \ref{4conCase:transitive-3}, \ref{4conCase:transitive-4}, \ref{4conCase:transitive-5}, then~$G$ is vertex-transitive.
    If~\ref{4conCase:transitive-2} $G$ is the $K_4$-expansion of a quasi-5-connected arc-transitive graph, then~$G$ is vertex-transitive by \cref{lem:expansion-edge-transitive}.
    Both~\ref{4conCase:transitive-6} a cycle of~$K_{2,2}$-bags and~\ref{4conCase:transitive-7} the line graph of the cube are arc-transitive.
    If~$G$ is a $K_4$-expansion of~\ref{4conCase:transitive-6} or~\ref{4conCase:transitive-7}, then~$G$ is vertex-transitive by \cref{lem:expansion-edge-transitive}.    
\end{proof}

\section{Analysing 3-regular quasi-4-connectivity}\label{sec:3regQ4C}

The \defn{line graph $L = L(G)$} of a graph~$G$ is the graph on the vertex set $V(L) := E(G)$ such that two vertices in~$L$ are adjacent if and only if their corresponding edges in~$G$ share an endvertex.

A graph~$G$ is \defn{internally-4-connected} if it is 3-connected and every 3-separation has a side that induces a claw. Note that internally-4-connected graphs are quasi-4-connected.

\begin{keylemma}
\label{lem:transitive-reduction-to-int-4-con}
    Let $G$ be a 3-regular quasi-4-connected graph.
    Then either $G$ is internally-4-connected and has girth~$\geq 4$, or $G=K_3\square K_2$.
\end{keylemma}
\begin{proof}
    Suppose that~$G$ is not internally-4-connected.
    It suffices to show that either $G$ is a cycle of $C_4$-bags of length~3 or~$G$ has girth~$\geq 4$.
    Note that $G$ has $\ge 5$ vertices as $G\neq K_4$.
    There is a vertex~$v$ in~$G$ of degree~3 that has two adjacent neighbours~$a$ and~$b$.
    Let~$c$ be the third neighbour of~$v$.

    First, we show that $ac,bc \notin E(G)$.
    Assume~$ac \in E(G)$. Then the pair $(\{v,a,b,c\},V(G)\sm\{v,a\})$ is a 2-separation with separator $\{b,c\}$ since $|V(G)| \geq 5$, contradicting that~$G$ is 3-connected.

    It follows that $(A,B) := (\{v,a,b\},V(G)\sm\{v,a,b\})$ is a 3-edge-cut. Since~$G$ is 3-connected and since the set $N(\{a,b\})$ separates $\{a,b\}$ from~$c$, we must have $|N(\{a,b\})| \geq 3$.
    In particular, no two edges in~$S(A,B)$ share an endvertex.
    We define~$a',b'$ and~$v'$ such that $S(A,B)=\{aa',bb',vv'\}$.
    Note that $N(\{a,b\})=\{a',b',v\}$ separates $\{a,b\}$ from $V(G) \sm \{v,a,b,a',b'\}$. By quasi-4-connectivity of~$G$, it follows that $V(G)=\{v,a,b,v',a',b'\}$. And by 3-regularity of~$G$, it follows that $v'a'b'$ is a triangle. 
    So, $G=K_3\square K_2$.
\end{proof}

Recall that $E(v)$ denotes the set of all edges incident with a vertex $v$ in a graph~$G$.

\begin{lemma}
\label{lem:transitive-cut-to-separation}
    Let~$G$ be a connected graph and let $L=L(G)$ be the line graph of~$G$.
    If $(A,B)$ is an edge-cut in~$G$ with $E(G[A]) \neq \emptyset \neq E(G[B])$, then \[(A',B') := (E(G[A])\cup S(A,B)\:,\: E(G[B]) \cup S(A,B))\] is a genuine separation of~$L$. 
    If $(A,B)$ has order~$k$, then $(A',B')$ has order~$k$ as well.
\end{lemma}
\begin{proof}
    Since $E(G)=E(G[A]) \,\dot\cup\, E(G[B]) \,\dot\cup\, S(A,B)$, we get $A' \cup B' = V(L)$ and $A' \cap B' = S(A,B)$.
    Note that for every vertex $v \in A$ we have $E(v) \subseteq A'$.
    Similarly, for every vertex $v \in B$ we have $E(v) \subseteq B'$.
    It follows that there is no $(A' \sm B')$--$(B' \sm A')$ edge in~$L$.
    From $E(G[A]) \neq \emptyset \neq E(G[B])$ we get that $(A',B')$ is proper.
    Hence, $(A',B')$ is a separation of~$L$ that has the same order as~$(A,B)$.

\end{proof}

\begin{lemma}
\label{lem:transitive-separation-to-cut}
    Let~$G$ be a connected graph and let $L=L(G)$ be the line graph of~$G$.
    If $(A',B')$ is a genuine separation of~$L$, then \[(A,B) := (\{v \in V(G) \mid E(v) \subseteq A'\}\:,\: \{v \in V(G) \mid E(v) \subseteq B'\})\] is a mixed-separation of~$G$ with $A \cap B = \{v \in V(G) \mid E(v) \subseteq A' \cap B'\}$.
\end{lemma}
\begin{proof}
    Note that for all $v \in V(G)$, the set $E(v) \subseteq V(L)$ induces a clique in~$L$.
    Hence, $E(v) \subseteq A'$ or $E(v) \subseteq B'$ for all $v \in V(G)$.
    From the definition of~$(A,B)$, it follows that $A \cup B = V(G)$. 
    Since $(A',B')$ is proper, it follows that there exist vertices $v_A,v_B \in V(G)$ with $E(v_A) \subseteq A' \nsupseteq E(v_B)$ and $E(v_B) \subseteq B' \nsupseteq E(v_A)$; that is, $v_A \in A \sm B$ and $v_B \in B \sm A$.
    Hence, $(A,B)$ is a mixed-separation.
    Moreover, a vertex~$v$ is contained in $A \cap B$ if and only if $E(v) \subseteq A' \cap B'$.
\end{proof}

\begin{lemma}
\label{lem:transitive-properties-line-graph}
    Let~$G$ be a 3-regular internally-4-connected graph and let~$L=L(G)$ be the line graph of~$G$.
    Then all of the following assertions hold:
    \begin{enumerate}
        \item\label{itm:transitive-properties-line-graph-0} $L$ is 4-regular and every vertex~$v \in V(L)$ is contained in exactly two triangles~$\Delta_1$ and~$\Delta_2$ in~$L$. Moreover, $V(\Delta_1 \cap \Delta_2)=\{v\}$.
        \item\label{itm:transitive-properties-line-graph-1} $L$ is 4-connected.
        \item\label{itm:transitive-properties-line-graph-3} For every triangle~$\Delta \subseteq L$ there is a vertex~$v$ in~$G$ such that~$E(v)=V(\Delta)$.
        \item\label{itm:transitive-properties-line-graph-4} For every edge $e \in E(L)$ there is a unique triangle~$\Delta(e)$ in~$L$ that contains~$e$.
    \end{enumerate}
\end{lemma}
\begin{proof}
    \ref{itm:transitive-properties-line-graph-0} follows from the definition of the line graph and from~$G$ having girth~$\geq 4$.
    
    \ref{itm:transitive-properties-line-graph-1}.
    Suppose for a contradiction that there is a genuine $(\leq 3)$-separation $(A',B')$ of~$L$.
    Let
    \[(A,B) := (\{v \in V(G) \mid E(v) \subseteq A'\}\:,\: \{v \in V(G) \mid E(v) \subseteq B'\})~,\]
    which is a mixed-separation of~$G$ by \cref{lem:transitive-separation-to-cut}.
    
    First, suppose $A \cap B \neq \emptyset$. 
    Then~$A' \cap B'$ induces a~$K_3$ in~$L$.
    By~\ref{itm:transitive-properties-line-graph-0}, every vertex~$v$ of~$L$ is contained in exactly two triangles, which only intersect in~$v$.
    It follows that~$(A,B)$ is a 1-separation of~$G$, a contradiction.
    
    Hence, we may assume that~$(A,B)$ is an edge-cut.
    Let $e=uv \in S(A,B)$, say with $u \in A$ and $v \in B$.
    Then $E(u) \subseteq A'$ and $E(v) \subseteq B'$ by definition of~$(A,B)$.
    Let~$\Delta_u$ and~$\Delta_v$ be the triangles in~$L$ induced by~$E(u)$ and~$E(v)$, respectively.
    From $V(\Delta_u) \subseteq A'$ and $V(\Delta_v) \subseteq B'$ and $V(\Delta_u \cap \Delta_v) = \{e\}$ (by~\ref{itm:transitive-properties-line-graph-0}) it follows that $e \in S(A',B')$.
    Hence, $S(A,B) \subseteq S(A',B')$; in particular~$(A,B)$ has order~$\leq 3$.
    Since~$G$ is internally-4-connected,~$(A,B)$ is an atomic edge-cut with $S(A,B)=S(A',B')$.
    Now $(A',B')$ is not proper, a contradiction.

    \ref{itm:transitive-properties-line-graph-3}. 
    From~$G$ being 3-regular and internally-4-connected, it follows that~$G$ has girth~$\geq 4$. 
    If~$\Delta$ is a triangle in~$L$ such that there is no vertex~$v$ in~$G$ with $V(\Delta) = E(v)$, then $V(\Delta)$ is the set of edges of a triangle in~$G$, contradicting that $G$ has girth~$\geq 4$.

    \ref{itm:transitive-properties-line-graph-4}. 
    Every edge $e_1 e_2\in E(L)$ is contained in at least one triangle: consider the endvertex $v\in G$ shared by $e_1$ and~$e_2$, then $E(v)$ induces a triangle in $L$ that contains the edge~$e_1 e_2$.
    Assume for a contradiction that there is an edge $e_1 e_2 \in E(L)$ that is contained in at least two triangles~$\Delta_1$ and~$\Delta_2$ of~$L$. By~\ref{itm:transitive-properties-line-graph-3}, there are two vertices $v_1, v_2 \in V(G)$ such that $E(v_i)=V(\Delta_i)$ for $i \in \{1,2\}$. Then $e_1,e_2 \in E(v_1) \cap E(v_2)$, contradicting that~$G$ is simple.
\end{proof}

\begin{lemma}
\label{lem:transitive-genuine-seps}
    Let~$G$ be a finite, 3-regular, internally-4-connected, vertex-transitive graph, and let $L=L(G)$ be its line graph.
    Let $(A',B')$ be a mixed-4-separation of~$L$ with $|A' \sm B'| \geq 2$ and $|B' \sm A'| \geq 2$.
    Then $S(A',B') \subseteq V(L)$.
    In particular, all tetra-separations of~$L$ are genuine 4-separations.
\end{lemma}
\begin{proof}
    First, observe that we can assume $G \neq K_4$.
    By \cref{lem:transitive-properties-line-graph},~$L$ is 4-connected.
    Let~$(A',B')$ be a a mixed-4-separation of~$L$ with $|A' \sm B'| \geq 2$ and $|B' \sm A'| \geq 2$.
    The left-right-reduction $(A'',B'')$ of~$(A',B')$ is a tetra-separation by~\cite[Key Lemma~5.10]{TetraDecomp}.
    Moreover, it satisfies $A'' \sm B'' \supseteq A' \sm B'$ and $B'' \sm A'' \supseteq B' \sm A'$ by~\cite[Observation~5.9]{TetraDecomp}, and every edge in~$S(A',B')$ is in~$S(A'',B'')$. Hence, it suffices to show that $S(A'',B'') \subseteq V(L)$.
    
    Assume for a contradiction that $e_1 \in S(A'',B'')$ is an edge in~$L$. 
    By \cref{lem:transitive-properties-line-graph}~\ref{itm:transitive-properties-line-graph-4}, there is a unique triangle~$\Delta_1$ in~$L$ that contains~$e_1$.
    Let~$v$ be the third vertex in $\Delta_1$ besides the endvertices of~$e_1$.
    Since~$\Delta_1$ is a triangle with the edge $e_1$ contained in $S(A'',B'')$, we get that~$v$ is contained in the separator $S(A'',B'')$ as well.
    By \cref{lem:transitive-properties-line-graph}~\ref{itm:transitive-properties-line-graph-0}, there is a triangle $\Delta_2 \neq \Delta_1$ in~$L$ such that $\Delta_1$ and $\Delta_2$ meet exactly in the vertex~$v$. 
    Let~$e_2$ denote the edge in $\Delta_2$ whose endvertices are distinct from~$v$.
    Since $L$ is 4-regular, the degree-condition for $v\in S(A'',B'')$ yields $e_2 \in S(A'',B'')$.
    So we know three elements of $S(A'',B'')$: they are $e_1,v$ and~$e_2$.
    The fourth element~$u$ of~$S(A'',B'')$ must be a vertex, as otherwise we could argue as for~$e$ to find a third triangle contributing both an edge and a vertex to~$S(A'',B'')$, contradicting that~$\Delta_1$ and~$\Delta_2$ already use all four edges incident to~$v$.
    By \cref{lem:transitive-properties-line-graph}~\ref{itm:transitive-properties-line-graph-0}, there are two triangles $\Delta,\Delta'$ in~$L$ that meet exactly in the vertex~$u$.
    By 4-regularity of $L$ and the degree-condition for $u\in S(A',B')$, we have $\Delta-u\se G[A'\sm B']$ and $\Delta'-u\se G[B'\sm A']$, say.

    Now consider the mixed-separation $(A,B)$ of $G$ that is induced by~$(A'',B'')$ as in \cref{lem:transitive-separation-to-cut}.
    The separator of~$(A,B)$ consists of two vertices $x_1,x_2$ that correspond to the triangles $\Delta_1,\Delta_2$, respectively, and of one edge~$f$ that corresponds to the vertex~$u$ (by \cref{lem:transitive-properties-line-graph}~\ref{itm:transitive-properties-line-graph-3}).
    In particular, $(A,B)$ is a mixed-3-separation of~$G$.
    Also note that $x_1 x_2 \in E(G)$ as the triangles $\Delta_1,\Delta_2$ share the vertex~$u$.

    Since $G \neq K_4$, there is a proper side of~$(A,B)$, say~$B \sm A$, that contains at least two vertices. Let~$x_3$ be the endvertex of~$f$ in $B \sm A$. Note that $(A \cup f, B)$ is a genuine 3-separation with separator~$\{x_1,x_2,x_3\}$.
    Since~$G$ is internally-4-connected, it follows that $x_1 x_2 \notin E(G)$, a contradiction.
\end{proof}

Let~$G$ be a graph.
A \defn{tetra-cut} of~$G$ is a 4-edge-cut of~$G$ that satisfies the matching-condition.
A tetra-cut of~$G$ is \defn{totally-nested} if it is nested with every tetra-cut of~$G$.

\begin{observation}
\label{obs:edge-cut-genuine}
    If~$G$ is a 2-connected graph and $(A,B)$ is a non-atomic edge-cut of~$G$ then $E(G[A]) \neq \emptyset \neq E(G[B])$.\qed
\end{observation}

\begin{lemma}\label{tetracutBijTetrasep}
    Let~$G$ be a finite, 3-regular, internally-4-connected, vertex-transitive graph and let $L=L(G)$ be its line graph.
    Let the map $(A,B)\mapsto (A',B')$ send non-atomic edge-cuts $(A,B)$ of~$G$ to genuine separations $(A',B')$ of~$L$ as in the statement of \cref{lem:transitive-cut-to-separation}.
    Then this map restricts to a bijection from the set of tetra-cuts of $G$ to the set of tetra-separations of~$L$.
\end{lemma}
\begin{proof}
    First note that $E(G[A]) \neq \emptyset \neq E(G[B])$ for every nonatomic edge-cut~$(A,B)$ of~$G$ by \cref{obs:edge-cut-genuine}.
    Hence, the prerequisites of \cref{lem:transitive-cut-to-separation} are fulfilled.
    Since $G$ has no isolated vertices, the map $(A,B)\mapsto (A',B')$ is injective.

    We claim that every tetra-cut $(A,B)$ is mapped to a tetra-separation~$(A',B')$.
    We get from \cref{lem:transitive-cut-to-separation} that $(A',B')$ is a 4-separation of~$L$.
    So it remains to check the degree-condition for the vertices in the separator of $(A',B')$.
    For this, let $v$ be an arbitrary vertex in $S(A',B')$.
    By \cref{lem:transitive-properties-line-graph}~\ref{itm:transitive-properties-line-graph-0}, the vertex $v$ is contained in exactly two triangles~$\Delta_1,\Delta_2$, and the two triangles meet precisely in the vertex~$v$.
    Each triangle~$\Delta_i\se L$ corresponds to an endvertex $x_i$ of~$v$ viewed as an edge of~$G$.
    Since $(A,B)$ is a tetra-cut, we have $x_1\in A\sm B$ and $x_2\in B\sm A$, say.
    Each $\Delta_i$ determines a 3-star in~$G$ with centre~$x_i$, and the edges in $S(A,B)$ form a matching from which each 3-star already contains $v$ (viewed as an edge of~$G$).
    Hence $\Delta_1\se L[A']$ and $v$ is the only vertex of $\Delta_1$ contained in~$S(A',B')$.
    So $v$ has $\ge 2$ neighbours in $L[A'\sm B']$.
    Similarly, $v$ has $\ge 2$ neighbours in $L[B'\sm A']$.
    Therefore, $(A',B')$ is a tetra-separation of~$L$.

    To conclude the proof, it remains to show that every tetra-separation of~$L$ equals $(A',B')$ for some tetra-cut~$(A,B)$.
    For this, let $(C,D)$ be an arbitrary tetra-separation of~$L$.
    By \cref{lem:transitive-genuine-seps}, $(C,D)$ is a genuine 4-separation.
    Let~$(A,B)$ be the mixed-separation in~$G$ as defined in \cref{lem:transitive-separation-to-cut} applied to~$(C,D)$.
    Then we claim that~$(A,B)$ is a tetra-cut and~$(A',B')=(C,D)$.
    By the degree-condition and \cref{lem:transitive-properties-line-graph}~\ref{itm:transitive-properties-line-graph-0}, for every $v \in S(C,D)$ there are two triangles~$\Delta_1(v), \Delta_2(v)$ with $\Delta_1(v) \sm \{v\} \subseteq C \sm D$ and $\Delta_2(v) \sm \{v\} \subseteq D \sm C$.
    That means, for two distinct $v,w \in S(C,D)$ we have $\Delta_i(v) \neq \Delta_i(w)$ for $i \in\{1,2\}$.
    With \cref{lem:transitive-properties-line-graph} it follows that~$(A,B)$ is a tetra-cut.
    Moreover, if $v \in C \sm D$ then $v \in E(G[A])$, and it follows that $v \in A' \sm B'$.
    Analogously, $v \in D \sm C$ implies $v \in B' \sm A'$.
    On the other hand, if $v \in C \cap D$ then $v \in S(A,B)$, and it follows that $v \in A' \cap B'$.
    Hence, $(A',B')=(C,D)$.
\end{proof}

\begin{lemma}
\label{lem:line-graph-preserves-crossing}
    Let~$G$ be a finite, 3-regular, internally-4-connected, vertex-transitive graph and let $L=L(G)$ be its line graph.
    Let $(A,B)\mapsto (A',B')$ be the bijection from the set of tetra-cuts~$(A,B)$ of~$G$ to the set of tetra-separations~$(A',B')$ of~$L$.
    Then $(A,B) \leq (C,D)$ if and only if $(A',B') \leq (C',D')$.
\end{lemma}
\begin{proof}
    If $(A,B) \leq (C,D)$ then $A \subseteq C$.
    Hence, $E(G[A]) \subseteq E(G[C])$ and $S(A,B) \subseteq E(G[C]) \cup S(C,D)$.
    It follows that $A' \subseteq C'$. Analogously, $B' \supseteq D'$.

    On the other hand, assume $(A',B') \leq (C',D')$.
    Let~$v$ be a vertex in~$A$.
    Recall that~$E(v)$ denotes the edges incident with~$v$ in~$G$.
    Then $E(v) \subseteq A' \subseteq C'$.
    From $E(v) \subseteq C'$ it follows that $v \in C$.
    Hence $A \subseteq C$ and, analogously, $B \supseteq D$.
\end{proof}

\begin{lemma}\label{totallynestedTetracut4cycle}
    Let~$G$ be a finite 3-regular internally-4-connected vertex-transitive graph and let~$(A_1,A_2)$ be a totally-nested tetra-cut in~$G$.
    Then $G[A_1]=C_4$ or~$G[A_2]=C_4$.
\end{lemma}
\begin{proof}    
    Let~$M$ be the union of the $\Aut(G)$-orbits of $(A_1,A_2)$ and $(A_2,A_1)$. Let $(C,D)$ be $\leq$-minimal in~$M$; that is, $(C',D') \not< (C,D)$ for every $(C',D') \in M$.

    We claim that every vertex in~$C$ is incident to an edge in~$S(C,D)$.
    Assume for a contradiction that $v\in C$ is not.
    Let $u\in C\sm D$ be an endvertex of one of the four edges in $S(C,D)$ and let $\phi \in \Aut(G)$ such that $\varphi(u)=v$.
    Then $(\varphi(C),\varphi(D))<(C,D)$ or $(\varphi(D),\varphi(C))<(C,D)$, contradicting the $\leq$-minimality of~$(C,D)$.
    
    Finally, since $G$ is internally-4-connected, every vertex has degree $\ge 3$, so $G[C]$ is a~$C_4$.
    The tetra-cut $(A_1,A_2)$ shares the above properties of~$(C,D)$ since $(C,D)\in M$.
\end{proof}

\begin{lemma}
\label{obs:i4c-no-3-edge-cut}
    Let~$G$ be an internally-4-connected graph.
    Then every $(\leq 3)$-edge-cut of~$G$ is an atomic 3-edge-cut.
\end{lemma}
\begin{proof}
    Since~$G$ is 3-connected, it does not contain a $(\leq 2)$-edge-cut.
    Assume that~$G$ has a 3-edge-cut $(A,B)$ with separator $S(A,B)=\{e_1, e_2, e_3\}$. Let~$a_i$ be the endvertex of~$e_i$ in~$A$, and let~$b_i$ be the endvertex of~$e_i$ in~$B$. 
    
    First, assume that~$(A,B)$ satisfies the matching-condition. 
    Since $\{a_1,a_2,b_3\}$ is a 3-separator separating~$a_3$ from~$\{b_1,b_2\}$, we must have $a_1 a_2 \notin E(G)$ by~$G$ being internally-4-connected.
    Since $\{b_1,a_2,a_3\}$ is a 3-separator separating~$a_1$ from~$\{b_2,b_3\}$, we must have $a_1 a_2 \in E(G)$ by~$G$ being internally-4-connected.
    This is a contradiction.

    Second, assume that~$(A,B)$ does not satisfy the matching-condition.
    If~$(A,B)$ is not atomic, we have without loss of generality $a_1 = a_2 \neq a_3$ and $b_1 \neq b_3$.
    Then, however, $\{a_1,b_3\}$ is a 2-separator separating~$a_3$ from~$b_1$.
    Hence,~$(A,B)$ is atomic.
\end{proof}

\begin{lemma}\label{tetracutCrossing}
    Let~$G$ be an internally-4-connected graph.
    Let~$(A,B)$ and~$(C,D)$ be two crossing tetra-cuts of~$G$.
    Then they cross with all links of size~2 and empty centre.
\end{lemma}
\begin{proof}
    Recall that $L(X,Y)$ denotes the corner-separator for the corner~$XY$.
    Note that, since~$(A,B)$ and~$(C,D)$ are edge-cuts, their corner-separators are edge-cuts as well.
    
    First assume that the centre is nonempty.
    Say, there is a diagonal edge from the $AD$-corner to the $BC$-corner.
    By \cite[Lemma 1.3.4]{Tridecomp}, either $|L(A,C)| \leq 3$ or $|L(B,D)| \leq 3$, say the former. By \cref{obs:i4c-no-3-edge-cut}, $L(A,C)$ is an atomic 3-edge-cut, contradicting that~$(A,B)$ and~$(C,D)$ satisfy the matching-condition.

    Similarly, we show that every link has size~2.
    Otherwise, there would exist a corner-separator of order~$\leq 3$, say $|L(A,C)| \leq 3$.
    By \cref{obs:i4c-no-3-edge-cut}, $L(A,C)$ is an atomic 3-edge-cut, contradicting that~$(A,B)$ and~$(C,D)$ satisfy the matching-condition.
\end{proof}

\begin{lemma}
\label{lem:transitive-links-of-size-2}
    Let~$G$ be a 3-regular internally-4-connected graph and let $L=L(G)$ be its line graph.
    Let~$(A',B')$ and~$(C',D')$ be two crossing tetra-separations of~$L$.
    Then they cross with all links of size~2 and empty centre.
\end{lemma}
\begin{proof}
    We combine \cref{tetracutBijTetrasep} with \cref{tetracutCrossing}.
\end{proof}

\begin{lemma}
\label{lem:transitive-tutte-bagel-3-conn}
    Let~$G$ be a 3-regular internally-4-connected vertex-transitive finite graph and let $L=L(G)$ be its line graph.
    Let~$(A',B')$ and~$(C',D')$ be two tetra-separations of~$L$ that cross with all links of size~2.
    Let~$\cO=(O,\cG)$ be the Tutte-bagel of~$L$ induced by~$(A',B')$ and~$(C',D')$.
    Then every torso of~$\cO$ is 3-connected and has $\geq 5$ vertices.
\end{lemma}
\begin{proof}
    We may assume the Tutte-bagel setting~\cite[9.12]{TetraDecomp} for~$L$, but with $(A',B')$ in place of $(A,B)$ and $(C',D')$ in place of~$(C,D)$.
    
    Assume that there is a torso~$H_s$ in~$\cO$ that is a triangle or 4-cycle.
    Say, $H_s$ stems from the $A'C'$-Tutte-path.
    Since~$(A',B')$ and~$(C',D')$ are genuine separations by \cref{lem:transitive-genuine-seps}, no link contains an edge.
    It follows that all conditions for the Link-Moving Lemma for vertices~\cite[9.44]{TetraDecomp} are fulfilled.
    By the Link-Moving Lemma for vertices~\cite[9.44~(i)]{TetraDecomp}, we get a tetra-separation of~$L$ that contains at least one edge in its separator.
    This contradicts~\cref{lem:transitive-genuine-seps}.

    Therefore, every torso~$H_s$ of~$\cO$ is 3-connected.
    Assume for a contradiction that there is a torso~$H_s$ of~$\cO$ with $\leq 4$ vertices.
    Then $H_s = K_4$.
    By \cref{lem:transitive-properties-line-graph}~\ref{itm:transitive-properties-line-graph-4}, every edge $e \in E(G_s)$ is contained in a unique triangle.
    If some adhesion set of~$s$ induces an edge~$e$, then~$e$ is contained in two triangles.
    Otherwise, the edges in~$G_s$ are not contained in any triangle.
    This is a contradiction.
\end{proof}

Let~$G$ be a graph.
A \defn{cycle-partition} of~$G$ is a pair $\cO=(O,(V_t)_{t \in V(O)})$ consisting of a cycle~$O$ and vertex sets~$V_t$ such that 
\begin{enumerate}
    \item $\{\,V_t \mid t \in V(O)\,\}$ is a partition of~$V(G)$; and
    \item every edge~$e \in E(G)$ either has its ends in a unique $V_t$ or there is an edge $tt' \in E(O)$ such that~$e$ has its ends in~$V_t$ and~$V_{t'}$.
\end{enumerate}
We call the sets $V_t$ the \defn{bags} of~$\cO$.
For every $tt' \in E(O)$, we call $E_G(V_t,V_{t'})$ the \defn{adhesion-set} of~$tt'$.

\begin{lemma}
\label{lem:transitive-cycle-partition}
    Let~$G$ be a 3-regular internally-4-connected vertex-transitive finite graph.
    Assume that~$G$ has two crossing tetra-cuts.
    Then~$G$ admits a cycle-partition~$\cO=(O,(V_t)_{t \in O})$ with $|O| \geq 4$ such that all of the following are satisfied:
    \begin{enumerate}
        \item\label{itm:transitive-cycle-partition-1} All adhesion-sets of~$\cO$ are 2-matchings;
        \item\label{itm:transitive-cycle-partition-2} Every bag of~$\cO$ induces a~$C_4$ or a~$K_2$;
        \item\label{itm:transitive-cycle-partition-3} Let $t \in V(O)$ and let $t_1,t_2$ be its neighbours on~$O$. Let~$U_i$ be the set of endvertices of edges in the adhesion-set for~$tt_i$ in~$V_t$. If~$V_t$ induces a 4-cycle, then this 4-cycle alternates between~$U_1$ and~$U_2$.
    \end{enumerate}
\end{lemma}
\begin{proof}
    Let~$L=L(G)$ be the line graph of~$G$.
    Let~$(A,B) \mapsto (A',B')$ be the bijection from \cref{tetracutBijTetrasep} between tetra-cuts $(A,B)$ in~$G$ and tetra-separations $(A',B')$ in~$L$.
    By \cref{lem:line-graph-preserves-crossing},~$(A,B)$ is a totally-nested tetra-cut of~$G$ if and only if~$(A',B')$ is a totally-nested tetra-separation of~$L$.

    Let~$(A,B)$ and $(C,D)$ be two crossing tetra-cuts in~$G$.
    By \cref{lem:line-graph-preserves-crossing} and \cref{lem:transitive-links-of-size-2}, the tetra-separations $(A',B')$ and $(C',D')$ cross with links of size~2 and empty centre.
    Let~$\cO'=(O',(G'_t)_{t \in V(O')})$ be the Tutte-bagel of~$L$ induced by~$(A',B')$ and~$(C',D')$.
    We denote by~$H'_t$ the torso of~$t \in V(O')$ in~$\cO'$.
    By \cref{lem:transitive-tutte-bagel-3-conn}, every torso of~$\cO'$ is 3-connected on $\geq 5$ vertices.

    Recall that $E(v)$ denotes the set of edges incident with the vertex~$v$ in~$G$.
    We define~$\cO=(O,(V_t)_{t \in V(O)})$ with $O:=O'$ and $V_t := \{v \in V(G) \mid E(v) \subseteq V_t\}$.
    We show that~$\cO$ is a cycle-partition of~$G$.
    First, note that for every $v \in V(G)$, the set~$E(v)$ induces a triangle in the graph~$L$, which in turn is contained in a unique torso~$H'_t$ of~$\cO'$.
    Hence, $\{\,V_t \mid t \in V(O)\,\}$ is a partition of~$V(G)$.
    Second, let~$e$ be an edge in~$G$.
    Then~$e$ is a vertex in~$L$ which is contained in two triangles $\Delta_1(e)$ and $\Delta_2(e)$ in~$L$.
    If $\Delta_1(e)$ and $\Delta_2(e)$ are contained in the same torso of~$\cO'$, then both endvertices of~$e$ are contained in the same bag of~$\cO$.
    Otherwise, since every torso of~$\cO'$ is 3-connected the adhesion-sets of $\cO'$ are pairwise disjoint by~\cite[Lemma~9.4]{TetraDecomp}, so $\Delta_1(e)$ and $\Delta_2(e)$ are contained in two torsos of~$\cO'$, whose corresponding nodes $t_1, t_2 \in V(O')$ are neighbours on~$O'$.
    Then, the endvertices of~$e$ are contained in $V_{t_1}$ and $V_{t_2}$, where $t_1 t_2 \in E(O)$.
    This shows that~$\cO$ is a cycle-partition of~$G$.

    It remains to show that~\ref{itm:transitive-cycle-partition-1}--\ref{itm:transitive-cycle-partition-3} are satisfied for the cycle-partition~$\cO$ of~$G$.

    \ref{itm:transitive-cycle-partition-1}. 
    Let~$tt' \in E(O)$.
    We show that the adhesion-set of~$tt'$ is a 2-matching.
    First, note that the adhesion-set of~$tt'$ in~$\cO$ is equal to the adhesion-set of~$tt'$ in~$\cO'$.
    Since the latter consists of two vertices of~$L$, the former consists of two edges of~$G$.
    Since every torso of~$\cO'$ is 3-connected and~$L$ is 4-regular, it follows that the adhesion-set of~$tt'$ in~$\cO'$ is independent.
    Hence, the adhesion-set of~$tt'$ in~$\cO$ is a matching.
    
    \ref{itm:transitive-cycle-partition-2}.
    We show the following strengthening: 
    If the torso~$H'_t$ of~$\cO'$ is good, then~$V_t$ induces a~$C_4$ in~$G$, and otherwise~$V_t$ induces a~$K_2$ in~$G$.
    First, assume that~$H'_t$ is good.
    Then by~\cite[Key Lemma~9.41]{TetraDecomp}, $(U_t,W_t)=(U'_t,W'_t)$ as defined in the Tutte-bagel setting~\cite[9.12]{TetraDecomp} is a totally-nested tetra-separation.
    By \cref{lem:transitive-separation-to-cut} and \cref{lem:line-graph-preserves-crossing},~$(V_t, V(G) \sm V_t)$ is a totally-nested tetra-cut.
    It follows that~$V_t$ induces a~$C_4$ by \cref{totallynestedTetracut4cycle}.

    Now assume that~$H'_t$ is bad.
    Then~$V(H'_t)$ consists of exactly five vertices.
    Let~$v$ be the vertex in the interior of~$H'_t$.
    By 4-regularity of~$L$, the vertex~$v$ is adjacent to every vertex in~$H'_t - v$.
    Since every torso of~$\cO'$ is 3-connected, every vertex~$u$ in an adhesion-set of~$H'_t$ has degree~2 in~$G'_t$.
    It follows that the torso~$H'_t$ is a 4-wheel, that~$G'_t$ is the union of two triangles that meet exactly in~$v$, and that~$V_t$ induces a~$K_2$ in~$G$.
    
    \ref{itm:transitive-cycle-partition-3}.
    We denote the vertices in~$U_i$ by~$u_i$ and~$u_i'$.
    Assume for contradiction that~$V_t$ induces a 4-cycle~$Z$ that does not alternate between~$U_1$ and~$U_2$.
    As shown in the proof of \ref{itm:transitive-cycle-partition-2}, the tetra-cut $(V_t, V(G) \sm V_t)$ is totally-nested.
    Then without loss of generality $Z = u_1 u_1' u_2' u_2 u_1$.
    Then, however, $(V_{t_1} \cup U_1, V(G) \sm (V_{t_1} \cup U_1))$ is a tetra-cut that crosses $(V_t, V(G) \sm V_t)$, a contradiction.
\end{proof}

\begin{keylemma}\label{CrossingTetraCutOutcome}
    Let~$G$ be a 3-regular internally-4-connected vertex-transitive graph.
    Assume that~$G$ has two crossing tetra-cuts.
    Then~$G$ is a cycle of~$C_4$-bags of length~$\geq 4$, or a cycle alternating between~$K_{2,2}$-bags and $C_4$-torsos of length~$\geq 8$.
\end{keylemma}
\begin{proof}
    By \cref{lem:transitive-cycle-partition},~$G$ admits a cycle-partition~$\cO=(O, (V_t)_{t \in O})$ with $|O| \geq 4$ such that properties~\ref{itm:transitive-cycle-partition-1}--\ref{itm:transitive-cycle-partition-3} from \cref{lem:transitive-cycle-partition} are fulfilled.
    Let $O=:t_1 t_2 \ldots t_\ell t_1$.
    For $t=t_i \in V(O)$ we denote by~$U_{t,1}$ the set of endvertices of the edges in the adhesion-set for~$tt_{i-1}$ that lie in~$V_t$; and by~$U_{t,2}$ we denote the set of endvertices of the edges in the adhesion-set for~$tt_{i+1}$ that lie in~$V_t$.
    
    We show that either every bag~$V_t$ of~$\cO$ induces a~$K_2$ or every bag~$V_t$ of~$\cO$ induces a~$C_4$ alternating between~$U_{t,1}$ and~$U_{t,2}$. In the former case, it follows that~$G$ is a cycle of~$C_4$-bags of length~$\geq 4$, while in the latter case it follows that~$G$ is a cycle alternating between $K_{2,2}$-bags and~$C_4$-torsos of length~$\geq 8$.

    Assume for a contradiction that there is a bag of~$\cO$ that induces a~$K_2$, and that there is a bag of~$\cO$ that induces a~$C_4$.
    Then, without loss of generality, one of the following cases occurs.

    \casen{The bags for $t_1$, $t_2$ and $t_3$ induce a~$K_2$ while the bag for~$t_4$ induces a~$C_4$.} Then every vertex in~$V_{t_2}$ is contained in two~$C_4$'s of~$G$ while every vertex in~$V_{t_4}$ is contained in exactly one~$C_4$ of~$G$, contradicting vertex-transitivity.

    \casen{The bags for~$t_2$ and~$t_3$ induce a~$K_2$ while the bags for~$t_1$ and~$t_4$ induce a~$C_4$.}
    Let~$u \in V_{t_2}$.
    Then $u$ is contained in exactly two 5-cycles $Z,Z'$. The intersection $Z\cap Z'$ is a path with three edges, and $u$ is an internal vertex of this path. The neighbour of $u$ in $V_{t_1}$ also lies on $Z$ and $Z'$, but it is not an internal vertex of $Z\cap Z'$, contradicting vertex-transitivity.

    \casen{The bag for~$t_2$ induces a~$K_2$ while the bags for~$t_1$ and~$t_3$ induce a~$C_4$.}
    Then every vertex in~$V_{t_2}$ is contained in exactly four 5-cycles of~$G$, while every vertex in~$V_{t_1} \cap N(V_{t_2})$ is contained in at most three 5-cycles.
    This contradicts vertex-transitivity.
\end{proof}

Recall that a graph~$G$ is \defn{2-quasi-5-connected} if~$G$ is quasi-4-connected and every 4-separation~$(A,B)$ of~$G$ satisfies $|A \sm B| \leq 2$ or $|B \sm A| \leq 2$.

\begin{lemma}
\label{lem:2-quasi-5-con}
    Let~$G$ be a 3-regular quasi-4-connected graph.
    Assume that $G$ has no tetra-cut.
    Then $G$ is 2-quasi-5-connected.
\end{lemma}
\begin{proof}
    If~$G$ is not 2-quasi-5-connected, it has a 4-separation~$(A,B)$ with $|A \sm B| \geq 3$ and $|B \sm A| \geq 3$.
    Since~$G$ is quasi-4-connected, it follows that~$(A,B)$ is tight.
    Let~$S(A,B)=\{v_1,v_2,v_3,v_4\}$.

    First we show that if~$v_i$ only has one neighbour~$u_i$ in $A \sm B$ and~$v_j \neq v_i$ only has one neighbour~$u_j$ in $A \sm B$, then $u_i \neq u_j$.
    Indeed, otherwise $(A',B'):=((A\cup\{u_i\})\sm\{v_i,v_j\}, B \cup \{u_i\})$ is a 3-separation of~$G$ with $|A\sm B| \geq 2$ and $|B \sm A| \geq 2$, contradicting that~$G$ is quasi-4-connected.
    By the above argument, it follows that the left-right-reduction of~$(A,B)$ satisfies the matching-condition.
    
    Since~$G$ is 3-regular and~$(A,B)$ is tight, every vertex~$v_i \in S(A,B)$ either has exactly one neighbour in $A \sm B$, or~$v_i$ has exactly one neighbour in~$B \sm A$ and no neighbour in $A \cap B$.
    Therefore, the left-right-reduction of~$(A,B)$ is a 4-edge-cut and hence a tetra-cut of~$G$.
\end{proof}

\begin{example}
    The toroidal hex-grid on a torus shows that 2-quasi-5-connected in \cref{lem:2-quasi-5-con} is best possible.
\end{example}

\begin{lemma}\label{unique4cyclePartition}
    Let~$G$ be a 3-regular internally-4-connected vertex-transitive graph such that every tetra-cut of~$G$ is totally-nested.
    Then either~$G$ is 2-quasi-5-connected or there is a canonical partition~$\cP=\{\,A_i:i\in I\,\}$ of~$V(G)$ such that
    \begin{enumerate}
        \item each~$A_i$ induces a~$C_4$;
        \item the edges in~$G$ with ends in different classes of~$\cP$ form a perfect matching in~$G$.
    \end{enumerate}
\end{lemma}
\begin{proof}
    If $G$ has no tetra-cut, then $G$ is 2-quasi-5-connected by \cref{lem:2-quasi-5-con}.
    So we may assume that $G$ has at least one tetra-cut.
    Let $\sigma:=\{\,(A_i,B_i):i\in I\,\}$ be the set of all tetra-cuts of $G$ whose side $A_i$ induces a 4-cycle.
    Since all tetra-cuts of $G$ are totally-nested, $\sigma$ is a star of tetra-cuts.
    Hence the sets $A_i$ are pairwise disjoint.
    Since $G$ is vertex-transitive, we have $\bigcup_{i\in I}A_i=V(G)$.
    Therefore, $\cP:=\{\,A_i:i\in I\,\}$ is a canonical partition of $G$ into vertex-sets that induce 4-cycles.

    Let $M$ denote the set of all edges in $G$ with ends in different classes of~$\cP$.
    Every vertex $v$ of $G$ lies in exactly one $A_i$, and since $G$ is 3-regular, it follows that $v$ is incident to exactly one edge in $M$.
    Hence $M$ is a perfect matching.
\end{proof}

\begin{keylemma}\label{allTetraCutsTotallyNested}
    Let~$G$ be a 3-regular internally-4-connected vertex-transitive graph.
    Assume that every tetra-cut in~$G$ is totally-nested.
    Then~$G$ is one of the following:
    \begin{itemize}
        \item a cycle alternating between $K_{2,2}$-bags and $C_4$-torsos of length~6;
        \item one of the three $C_4$-expansions of the line graph of the cube shown in \cref{fig:LineGraphCube};
        \item a vertex-transitive $C_4$-expansion of a cycle of $K_{2,2}$-bags of length~$\ge 3$;
        \item a $C_4$-expansion of a quasi-5-connected 4-regular arc-transitive graph;
        \item 2-quasi-5-connected.
    \end{itemize}
\end{keylemma}
\begin{proof}
    Assume that $G$ is not 2-quasi-5-connected.
    Then $V(G)$ has a canonical partition $\cP$ into vertex-sets that induce 4-cycles by \cref{unique4cyclePartition}.
    Fix $U\in\cP$ and let $\{\,U_i:i\in I\,\}$ be the set of classes $U_i\in\cP$ that contain neighbours of~$U$.
    Let $M$ denote the set of edges of $G$ with ends in different classes of~$\cP$, so $M$ is a perfect matching by \cref{unique4cyclePartition}.
    Let $F\se M$ denote the set of edges leaving~$U$.

    \casen{Case: all four edges in $F$ end in the same~$U_i$.}
    Then $F=M$ and so $V(G)=U\cup U_i$.
    Both $U$ and $U_i$ induce 4-cycles, so either $G=C_4\square K_2$ which is a cycle of $C_4$-bags of length~4, or $G$ is the life belt.
    However, in both cases $G$ has two crossing tetra-cuts, so neither case occurs.

    \casen{Case: exactly three edges in $F$ end in the same~$U_i$.}
    This case is impossible: Then the vertex in $U$ and the vertex in $U_i$ that are not incident with these three edges form a 2-separator of~$G$, contradicting that $G$ is internally-4-connected.

    \casen{Case: there is a $U_i$ in which exactly two edges from $F$ end.}
    Then four edges from $M$ leave $U\cup U_i$.
    Hence $U\cup U_i$ is a side of a tetra-cut of~$G$.
    Since this tetra-cut is totally-nested by assumption, it has a side that induces a 4-cycle by \cref{totallynestedTetracut4cycle}.
    This side can only be $V(G)\sm (U\cup U_i)$, which then must be a $U_j$ for some $j\in I\sm\{i\}$.
    Now $U$ and $U_i$ each send exactly two edges to~$U_j$.
    If $G$ is a cycle alternating between $K_{2,2}$-bags and $C_4$-torsos of length~6, then we are done.
    Otherwise the 4-cycle $O$ induced by~$U$, say, first traverses the two neighbours $u_i,u_i'$ of $U_i$ in~$U$, and then traverses the two neighbours $u_j,u_j'$ of $U_j$ in~$U$.
    But then $(U,V(G-U))$ and $(U_i\cup \{u_i,u_i'\},U_j\cup\{u_j,u_j'\})$ are two crossing tetra-cuts in~$G$, contradicting our assumption.

    \casen{Case: $|I|=4$ and exactly one edge from $F$ ends in each~$U_i$.}
    Then $G/\cP$ is 4-regular and 4-connected.
    Moreover, $G/\cP$ is arc-transitive, as $\cP$ is unique and every vertex of $G$ is incident with a unique edge in~$F$.
    Then by \cref{cor:transitive-essentially-5-connected,cor:symmetric-crossing}, the graph $G/\cP$ is one of the following:
    \begin{itemize}
        \item the line-graph of the cube,
        \item a cycle of $K_{2,2}$-bags of length~$\ge 3$, or
        \item essentially-5-connected.
    \end{itemize}
    Assume first that $G/\cP$ is the line-graph of the cube, so $G$ is some $C_4$-expansion thereof.
    Since $G$ is vertex-transitive, there are only three possible $C_4$-expansions: these are precisely the three depicted in \cref{fig:LineGraphCube}.

    Assume finally that $G/\cP$ is essentially-5-connected.
    Then $G/\cP$ is quasi-5-connected by \cref{cor:symmetric-contract-essentially-5-connected}.
    Hence $G$ is the $C_4$-expansion of a quasi-5-connected 4-regular arc-transitive graph.
\end{proof}

\begin{proposition}\label{3regCase:transitive}
    A 3-regular quasi-4-connected finite graph $G$ is vertex-transitive if and only if $G$ is either
    \begin{enumerate}
        \item\label{3regCase:2Q5C} 2-quasi-5-connected and vertex-transitive,
        \item\label{3regCase:C4expQ5C} a vertex-transitive $C_4$-expansion of a quasi-5-connected 4-regular arc-transitive graph,
        \item\label{3regCase:C4bags} a cycle of~$C_4$-bags of length~$\geq 4$ or $K_3\square K_2$,
        \item\label{3regCase:K22C4} a cycle alternating between~$K_{2,2}$-bags and $C_4$-torsos of length~$\ge 6$,
        \item\label{3regCase:C4LineCube} one of the three $C_4$-expansions of the line-graph of the cube shown in \cref{fig:LineGraphCube}, or
        \item\label{3regCase:C4expK22rest} a vertex-transitive $C_4$-expansion of a cycle of $K_{2,2}$-bags of length~$\ge 3$.
    \end{enumerate}
\end{proposition}
\begin{proof}
    ($\Rightarrow$).
    By \cref{lem:transitive-reduction-to-int-4-con}, we may assume that $G$ is internally 4-connected and has girth~$\ge 4$.
    If $G$ has two crossing tetra-cuts, then we are done by \cref{CrossingTetraCutOutcome}.
    Otherwise every tetra-cut of $G$ is totally-nested.
    Then we are done by \cref{allTetraCutsTotallyNested}.

    ($\Leftarrow$) is straightforward.
\end{proof}

\begin{proof}[Proof of \cref{mainthm:transitive}]
    We combine \cref{4conCase:transitive} and \cref{3regCase:transitive}.
\end{proof}

\section{Analysing 3-regular arc-transitivity}\label{sec:4regArcTrans}

\begin{lemma}
\label{lem:C4-expansion-not-edge-transitive}
    Let~$k \leq 5$ and let~$G$ be a $C_k$-expansion of a $k$-regular graph.
    Then~$G$ is not edge-transitive.
\end{lemma}
\begin{proof}
    Let~$H$ be a $k$-regular graph and let~$G$ be a $C_k$-expansion of~$H$.
    Let~$\cP$ be the corresponding partition of $V(G)$.
    Then every edge induced by some $P \in \cP$ is contained in a $C_k$, while every edge not induced by any $P \in\cP$ is not contained in a $C_\ell$ for $\ell \leq 5$.
    So, $G$ is not edge-transitive.
\end{proof}

\begin{lemma}\label{lem:3-regular-arc-transitive}
     A 3-regular finite graph~$G$ is quasi-4-connected arc-transitive if and only if~$G$ is the cube or~$G$ is a 2-quasi-5-connected arc-transitive graph.
\end{lemma}
\begin{proof}
    ($\Rightarrow$).
    Since~$G$ is arc-transitive and has no isolated vertices, it is vertex-transitive (and additionally edge-transitive).
    So, we can apply \cref{mainthm:transitive}.
    Since~$G$ is 3-regular, we can exclude the graphs from \cref{mainthm:transitive}~\ref{mainitm:transitive-1}--\ref{mainitm:transitive-7} which are $(\geq 4)$-regular.
    By \cref{lem:C4-expansion-not-edge-transitive}, $G$ cannot be the $C_4$-expansion of a 4-regular graph.
    Hence, $G$ is one of the following (with the numbering from \cref{mainthm:transitive}):
    \begin{enumerate}
        \item[\ref{mainitm:3reg2q5c}] a 3-regular 2-quasi-5-connected arc-transitive graph, or
        \item[\ref{mainitm:transitive-K22C4}] a cycle alternating between $K_{2,2}$-bags and $C_4$-torsos of length~$\geq 6$,
        \item[\ref{mainitm:transitive-C4}] a cycle of $C_4$-bags of length~$\ge 4$ or $K_3\square K_2$.
    \end{enumerate}
    
    The graph $G$ cannot be a cycle alternating between $K_{2,2}$-bags and $C_4$-torsos of length~$\geq 6$, since it is not edge-transitive (some edges are contained in 4-cycles while others are not).
    Similarly,~$G$ cannot be a cycle of $C_4$-bags of length~$\ge 5$ (some edges are contained in two 4-cycles while others are not).
    Moreover,~$G \neq K_3\square K_2$ (some edges are contained in triangles while others are not).

    Hence, $G$ is a cycle of $C_4$-bags of length~$4$ (which is the cube), or~$G$ is a 3-regular 2-quasi-5-connected arc-transitive graph.

    ($\Leftarrow$) is trivial.
\end{proof}

\begin{proof}[{Proof of \cref{mainCor}}]
    We combine \cref{mainthm:transitive} with \cref{intro:triVxCon} and \cref{lem:3-regular-arc-transitive}.
\end{proof}

\section{Appendix: Analysing 3-connectivity}\label{sec:3Con}

For the short proof, we use the terminology of~\cite{TriAiC,Tridecomp}.

\begin{proposition}\label{tri-transitive}\cite[Corollary~2]{Tridecomp}.
    Every vertex-transitive finite connected $G$ is either
    \begin{itemize}
        \item 3-connected and every non-trivial strong tri-separation of $G$ is a 3-edge-cut with a side that induces a triangle,
        \item a cycle, $K_2$ or $K_1$.
    \end{itemize}
\end{proposition}

\begin{proposition}\label{tri-transitive-better}
    Every vertex-transitive finite connected graph $G$ is either
    \begin{itemize}
        \item internally 4-connected and vertex-transitive,
        \item the $K_3$-expansion of a quasi-4-connected 4-regular arc-transitive graph,
        \item a cycle, $K_3\square K_2$, $K_2$ or $K_1$.
    \end{itemize}
\end{proposition}
\begin{proof}
    By \cref{tri-transitive}, we may assume that $G$ is 3-connected and every non-trivial strong tri-separation of $G$ is a 3-edge-cut with a side that induces a triangle.
    If $G$ has no non-trivial strong tri-separation, then $G$ is internally-4-connected by \cite[{}1.2.8]{Tridecomp}.
    Hence we may assume that $G$ has at least one non-trivial strong tri-separation, and by assumption it has a side that induces a triangle.
    
    Let $\{\,(A_i,B_i):i\in I\,\}=\sigma$ be the set of all non-trivial strong tri-separations of~$G$ with $G[A_i]=K_3$.
    Since each $S(A_i,B_i)$ is a 3-edge-cut, the tri-separations $(A_i,B_i)$ are totally-nested by \cite[{}1.3.10]{Tridecomp}, so $\sigma$ is a star of tri-separations.
    Hence the $A_i$ are pairwise disjoint, and every vertex of $G$ is contained in some $A_i$ by vertex-transitivity, so $\cP:=\{\,A_i:i\in I\,\}$ is a partition of~$V(G)$.
    
    Since each $S(A_i,B_i)$ is a 3-matching, the set $M$ of cross-edges of~$\cP$ is a perfect matching.
    If two edges in $M$ join $A_i$ to $A_j$ for some $i\neq j\in I$, then the two vertices in $A_i\cup A_j$ that are not incident to these edges form a 2-separator of~$G$ contradicting 3-connectivity, unless $G=K_3\square K_2$.
    Hence we may assume that at most one edge in $M$ joins any $A_i$ to any $A_j$, so $G/\cP$ is 3-regular.

    Since $G$ is vertex-transitive and the $K_3$-expansion of $G/\cP$, it follows that $G/\cP$ is arc-transitive by \cref{lem:expansion-edge-transitive}.
    In particular, $G/\cP$ is vertex-transitive, and since $G/\cP$ is 3-regular we get that $G/\cP$ is 3-connected by \cref{tri-transitive}.
    However, $G/\cP$ has no non-atomic 3-edge-cut, as this would yield a non-trivial strong tri-separation of $G$ that is a 3-edge-cut but whose sides have size $\ge 9$, contradicting our assumption.
    Hence applying \cref{tri-transitive} to $G/\cP$ in combination with \cite[{}1.2.8]{Tridecomp} yields that $G/\cP$ is quasi-4-connected.
\end{proof}

\begin{proof}[Proof of \cref{intro:triVxCon}]
    \cref{tri-transitive-better} implies \cref{intro:triVxCon}.
\end{proof}

\begin{acknowledgment}
We thank Agelos Georgakopoulos for telling us about \cref{GodsilResult} of Godsil and Royle.
\end{acknowledgment}

\bibliographystyle{amsplain}
\bibliography{collective}
\end{document}